\documentclass[a4paper,oneside, reqno]{amsart}
\usepackage{mathrsfs}
\usepackage{amsfonts}
\usepackage{amssymb}
\usepackage{amsxtra}
\usepackage{dsfont}
\usepackage{bbm}
\usepackage{bm}
\usepackage{comment}
\usepackage{enumitem}

\usepackage{orcidlink}

\usepackage[english,polish]{babel}

\usepackage{color}

\theoremstyle{plain}
\newtheorem{theorem}{Theorem}[section]
\newtheorem{proposition}{Proposition}[section]

\newtheorem{lemma}[proposition]{Lemma}

\theoremstyle{definition}

\numberwithin{equation}{section}

\newcounter{thm}

\newcommand{\RR}{\mathbb{R}}
\newcommand{\ZZ}{\mathbb{Z}}
\newcommand{\TT}{\mathbb{T}}
\newcommand{\CC}{\mathbb{C}}
\newcommand{\NN}{\mathbb{N}}
\newcommand{\QQ}{\mathbb{Q}}

\newcommand{\ex}{\mathfrak{e}}

\newcommand{\R}{\mathbb{R}}
\newcommand{\Z}{\mathbb{Z}}

\newcommand{\C}{\mathbb{C}}
\newcommand{\N}{\mathbb{N}}

\newcommand{\ind}[1]{{\mathds{1}_{{#1}}}}
\newcommand{\dist}{\operatorname{dist}}
\newcommand{\lcm}{\operatorname{lcm}}

\newcommand{\abs}[1]{{\lvert {#1} \rvert}}

\newcommand{\dpi}{2\pi i}

\DeclareMathOperator{\supp}{supp}

\DeclareMathOperator{\loc}{loc}

\title[Discrete analogues in harmonic analysis: $TT^*$ methods]
{Discrete analogues in harmonic analysis: $TT^*$ methods}

\author[B. Langowski]{Bartosz Langowski \orcidlink{0000-0002-1516-1689}}
\address{Bartosz Langowski, \newline
Department of Mathematics and Physical Sciences, 
Franciscan University of Steubenville, 
1235 University Blvd., Steubenville, OH 43952, USA
}
\email{blangowski@franciscan.edu}

\author[Mariusz Mirek]{Mariusz Mirek \orcidlink{0000-0001-7641-9893} }
\address{Mariusz Mirek,     \newline
Department of Mathematics,
Rutgers University,
Piscataway, NJ 08854-8019, USA 
and
Instytut Matematyczny,
Uniwersytet Wroc{\l}awski,
Plac Grunwaldzki 2,
50-384 Wroc{\l}aw
Poland}
\email{mariusz.mirek@rutgers.edu}

\author[T.Z. Szarek]{Tomasz Z. Szarek \orcidlink{0000-0003-0821-5607}}
\address{Tomasz Z. Szarek,     \newline
Department of Mathematics, University of Georgia,
Athens, GA 30602, USA
and
Instytut Matematyczny,
Uniwersytet Wroc{\l}awski,
Plac Grunwaldzki 2,
50-384 Wroc{\l}aw,
Poland}
\email{tzs10705@uga.edu}

\begin{document}
\selectlanguage{english}

\begin{abstract}
In this note we present how the almost-orthogonality methods based on
$TT^*$ arguments can be employed to study boundedness of discrete
operators of Radon type. Almost-orthogonality methods have particular
significance when the classical
Fourier methods are not available. However here, to avoid
technicalities and present the key ideas behind the discrete
almost-orthogonality methods, we  give a new proof of the
$\ell^2(\mathbb{Z}^d)$-boundedness of Bourgain's maximal inequality
for Radon polynomial averages.
\end{abstract}

\thanks{Bartosz Langowski was supported by the AMS-Simons Research Enhancement Grant for Primarily Undergraduate Institution (PUI) Faculty. Mariusz Mirek was supported by NSF CAREER grant
DMS-2236493. Tomasz Z.\ Szarek was supported by
the Simons Foundation grant SFI-MPS-TSM-00013714 and by
the National Science Centre, Poland, grant Sonata Bis 2022/46/E/ST1/00036.}

\maketitle

\section{Introduction}
\label{sec:intro}

The purpose of this paper is to demonstrate how almost-orthogonality methods can be employed in the field of discrete analogues in harmonic analysis. However, we will not prove any new results. Instead, we illustrate how almost-orthogonality methods can be used to provide new proofs of some well-known results for discrete Radon operators, emphasizing the key points of this method and its wide applicability.

\subsection{A brief history}\label{subsec:main}

Almost-orthogonality methods have been developed  for decades. The key mechanism behind that methods is
a simple \textit{squaring identity} $|z|^2=z\overline{z}$ for complex
numbers $z\in\CC$. Surprising applications of the squaring identity arise in estimates for oscillatory integrals \cite{bigs} and exponential sums \cite{IK}. To see
this define $\ex(z):=e^{2\pi i z}$ for $z\in\CC$ and  consider the \textit{Gauss sum} given by
\begin{align*}
  G(a/q)=\frac{1}{q}\sum_{n=1}^{q}\ex(an^2/q).
\end{align*}
If $(a, q)=1$, then, using the squaring identity, one can show that $G(a/q)$ has a lot of cancellations, namely  $|G(a/q)|\le 2q^{-1/2}$. Indeed, squaring the expression $qG(a/q)$, note that
\begin{align*}
|qG(a/q)|^2=qG(a/q)q\overline{G(a/q)}=\sum_{m=1}^q\sum_{n=1}^q\ex(a(m^2-n^2)/q)
=
\sum_{d =1}^q\bigg(\sum_{n =1}^q\ex({2adn}/{q}) \bigg)\ex(ad^2/q).
\end{align*}
Analyzing the inner sum, we see that it equals $q$ if $2d \equiv 0 \ (\mathrm{mod}\ q)$ (which happens very rarely) and $0$ otherwise, proving that 
$|qG(a/q)|^2 = q (1 + \ind{2|q} (-1)^{q/2})$
and consequently $|qG(a/q)|^2 \le 2q$ as desired. Interestingly, this upper bound is sharp.

This simple idea of squaring was essential in Weyl's differencing method, which yielded Weyl's equidistribution theorem \cite{Weyl} for polynomials and secured Weyl's inequality as a key tool in estimates of exponential sums (see Lemma \ref{lem:A3}). Later, Weyl's differencing method was refined by van der Corput, see \cite[Chapter 8]{IK}, and today it is widely applicable, spanning number theory, additive combinatorics, ergodic theory, harmonic analysis, and even beyond.
One of the most elegant and general formulations of the Weyl--van der Corput differencing method in Hilbert spaces can be subsumed in the following proposition.
\begin{proposition}[Weyl--van der Corput differencing method in Hilbert spaces, see \cite{EW}]
\label{prop:1}
Let $(u_n)_{n \in \mathbb Z_+}$ be a bounded sequence in a Hilbert space
$(\mathbb H, \|\cdot\|)$. Define a sequence $(s_n)_{n \in \mathbb Z_+}$ of real numbers by
\[
s_h = \limsup_{N \rightarrow \infty} \Big| \frac{1}{N}\sum_{n=1}^N\langle u_{n+h}, u_n \rangle \Big|.
\]
If $\lim_{H \rightarrow \infty} \frac{1}{H}\sum_{h=1}^Hs_h = 0$,
then $\lim_{N \rightarrow \infty} \big\|\frac{1}{N}\sum_{n=1}^N u_n \big\| = 0$. 
\end{proposition}

Proposition \ref{prop:1} can be thought of as an almost-orthogonality,
or a kind of asymptotic orthogonality, criterion in Hilbert spaces. It
has important applications in modern ergodic theory. We refer to
\cite{Fra}, where precise references and additional information on
this topic can be found. We now use Proposition \ref{prop:1} to prove
Furstenberg's theorem \cite{Fur2}, asserting that ergodic averages
with polynomial iterates converge in norm. More precisely, let
$P\in\mathbb Z[\rm n]$ be a polynomial with integer
coefficients. Assuming that $(X, \mathcal B(X), \mu)$ is a probability
space equipped with a measure-preserving and invertible transformation $T \colon X\to X$,
which is totally ergodic (i.e. $T^n$ is ergodic for every
$n\in\mathbb Z_+$) one can prove that
\begin{align}
\label{eq:10}
\lim_{N\to\infty}\Big\|\frac{1}{N}\sum_{n=1}^Nf\circ T^{P(n)}-\int_Xfd\mu\Big\|_{L^2(X)}=0.
\end{align}
By linearity we can assume that  $\int_X fd\mu = 0$. To prove \eqref{eq:10} we proceed by induction on the degree of
$P$. If $P$ is linear, then the result follows from von Neumann's mean
ergodic theorem \cite{vN} and the total ergodicity of $T$. Now we can assume that
${\rm{deg}}(P) \geq 2$, and by Proposition \ref{prop:1} for
$\mathbb H=L^2(X)$ it suffices to prove that
$\lim_{N \rightarrow \infty} \frac{1}{N}\sum_{n=1}^N u_n = 0$ in $L^2(X)$
with $u_n = f \circ T^{P(n)}$.  Fix $h \in \mathbb Z_+$ and note that
$n\mapsto P(n+h) - P(n)$ is a polynomial of degree
${\rm{deg}}(P)-1$. By the induction hypothesis, we have
$\lim_{N \rightarrow \infty} \frac{1}{N}\sum_{n=1}^N f \circ T^{P(n+h) - P(n)} = 0$
in $L^2(X)$ for every $h \in \mathbb Z_+$. By the Cauchy--Schwarz
inequality
\[
\lim_{N\rightarrow \infty} \frac{1}{N}\sum_{n=1}^N \langle u_{n+h}, u_n \rangle =
\lim_{N\rightarrow \infty} \Big\langle \frac{1}{N}\sum_{n=1}^N f \circ T^{P(n+h) - P(n)}, f \Big\rangle = 0.
\]
This establishes the hypothesis of Proposition \ref{prop:1}, which
consequently implies \eqref{eq:10}. Furstenberg, used \eqref{eq:10} to
establish the polynomial Poincar{\'e} recurrence theorem \cite{Fur2}
and obtained an ergodic proof of S{\'a}rk{\"o}zy theorem \cite{Sark1,
Sark2}, which asserts that for every polynomial $P\in\mathbb Z[\rm n]$ with
$P(0)=0$ and any $A\subseteq\ZZ$ with positive
upper density there are integers $x,n\in\ZZ$, with $n \neq 0$, such that
$x, x+P(n)\in A$.

\medskip

The squaring identity $|z|^2=z\overline{z}$ for $z\in\CC$ has its Hilbert-space counterpart, the so-called $TT^*$ identity.  Namely, if $(\mathbb H, \|\cdot\|)$ is
a Hilbert space and $T \colon \mathbb H\to \mathbb H$ is a bounded linear operator on $\mathbb H$, then
\begin{align}
\label{eq:9}
\|TT^*\|_{\mathbb H\to \mathbb H}=\|T\|_{\mathbb H\to \mathbb H}^2,
\end{align}
where $\|\cdot\|_{\mathbb H\to \mathbb H}$ denotes the operator
norm. Identity \eqref{eq:9} has had a profound impact on many
developments in mathematics over the years. It is impossible to list
all references where the $TT^*$ methods were critical; however, their
role in harmonic analysis and number theory --- which is important
from the point of view of this paper --- remains indispensable.  In
particular, the $TT^*$ methods are decisive in estimates of
oscillatory integrals and exponential sums, and were key to the
Tomas--Stein restriction phenomenon, see \cite[Chapter IX]{bigs},
which led to the restriction conjecture, a major open problem in
harmonic analysis. More on these topics can be found on Tao's blog
\cite{T}. The $TT^*$ methods have also been extensively
exploited in number theory for decades. Recently, a variant of the
$TT^*$ method was decisive in a remarkable paper by Guth and Maynard
\cite{GM}, in which a breakthrough for a large-value problem for
Dirichlet polynomials --- a problem related to the zeros of the
Riemann zeta function --- was achieved. We also refer to Guth's survey
\cite{Guth}, in which $TT^*$ methods among the so-called large-value
estimates are discussed in harmonic analysis, number theory,
statistics, and computer science. Here, in addition to the problems
mentioned above, we emphasize two important results based on $TT^*$
methods that have become central in classical harmonic analysis and
will also play a significant role in this paper.

\medskip

The first one is a striking application of the $TT^*$ method from
\eqref{eq:9} by Kolmogorov and Seliverstov \cite{KS} and Plessner
\cite{P} in their work on pointwise convergence of Fourier series.  We now
illustrate their method in the context of Stein's observation
\cite{Steinmax} on the Hardy--Littlewood maximal inequality \cite{HLmax} on
$L^2(\RR)$. For any locally integrable function $f\in L^1_{\loc}(\RR)$, we define its average
\begin{align}
\label{eq:17}
M_tf(x):=\frac{1}{t}\int_0^tf(x-y)\,dy, \qquad x\in\RR, \quad  t>0,
\end{align}
and the corresponding maximal function $M_{*}f(x):=\sup_{t>0}|M_tf(x)|$. Taking a measurable function $N \colon \RR\to\RR_+$, we note that for any  measurable $f\ge0$, we have the following pointwise bound
\begin{align}
\label{eq:11}
M_{N(x)}M_{N(x)}^*f(x)
\le 
M_{N(x)}f(x)+M_{N(x)}^*f(x), \qquad x\in\RR.
\end{align}
Thus by \eqref{eq:11}, we obtain $\|M_{N(\cdot)}\|_{L^2(\RR) \to L^2(\RR)}^2\le 2\|M_{N(\cdot)}\|_{L^2(\RR) \to L^2(\RR)}$, which immediately yields the Hardy--Littlewood maximal inequality
\begin{align}
\label{eq:12}
\|M_{*}f\|_{L^2(\RR)}\le 2\|f\|_{L^2(\RR)},\qquad f\in L^2(\RR).
\end{align}
This approach should not be considered competitive with the
Hardy--Littlewood maximal inequality, since by \cite{HLmax} we know
that \eqref{eq:12} extends to \(L^p(\mathbb{R})\) spaces with
\(p \in (1, \infty]\), and it satisfies the weak type $(1,1)$
inequality for integrable functions \(f \in L^1(\mathbb{R})\). So we
obtain the best that can be expected for these kinds of questions.  We
instead want to emphasize that the $TT^*$ methods in inequalities like
\eqref{eq:11} exhibit a certain bootstrap phenomenon that may serve as
an alternative method for approaching maximal inequalities in
general. This may be especially important in situations where endpoint
estimates at obvious endpoints may not be available. These ideas were
further studied in the context of ergodic operators by Stein
\cite{Steinmax} and later by Weiss \cite{W}. More on the almost-orthogonality methods in the  ergodic
context  can be found in Nevo's survey \cite{Nevo}. Also, the idea of
employing $TT^*$ methods like in \eqref{eq:11} turned out to be
fruitful for polynomial Carleson-type operators by Stein and Wainger
\cite{SW}, whenever polynomials have no linear term. Far-reaching
extensions of these methods emerged in the understanding of maximally
modulated polynomial Carleson--Radon type operators by Pierce and Yung
\cite{PY}, by Guo, Pierce, Roos and Yung \cite{GPRY}, and recently by
Anderson, Maldague, Pierce, and Yung \cite{AMPY}. The discrete  maximally modulated polynomial Carleson type operators were studied
by Krause and Roos in \cite{KR}.

\medskip

The second application of the  $TT^*$ methods is a celebrated Cotlar--Knapp--Stein almost-orthogonality lemma, which provides a criterion for boundedness of sums
of linear operators in Hilbert spaces.

\begin{lemma}[Cotlar--Knapp--Stein almost-orthogonality lemma]
\label{lem:1}
Let $(T_j)_{j \in \ZZ}$ be a family of linear and bounded operators mapping a Hilbert space $(\mathbb H, \|\cdot\|)$ to itself. Assume that $\gamma \colon \ZZ \rightarrow \RR_+$ is a function such that 
\begin{equation}\label{eq:cotlarsteinhyp}
\|T_j^*T_k\|_{\mathbb H \to \mathbb H} + \|T_j T_k^*\|_{\mathbb H \to \mathbb H} \leq \gamma(j-k), \qquad j, k\in\ZZ,
\end{equation}
and suppose that     $A := \sum_{j \in \ZZ} \sqrt{\gamma(j)} < \infty$. Then the following three conclusions hold.
\begin{enumerate}
    \item For any finite subset $S \subseteq \ZZ$, we have
    \[ \Big\|\sum_{j \in S} T_j \Big\|_{\mathbb H \to \mathbb H} \leq A.\]
    \item For all $x \in \mathbb H$, we have
    \[ \sum_{j \in \ZZ} \|T_j x\|^2 \leq A^2 \|x\|^2.\]
    \item For all $x \in \mathbb H$, the sequence $\sum_{|j|\leq N} T_j(x)$ converges to some $T(x)$ as $N \rightarrow \infty$ in the norm topology of $\mathbb H$. The linear operator $T \colon \mathbb H\to \mathbb H$ defined in this way is bounded  with $\|T\|_{\mathbb H \to \mathbb H} \leq A$. 
\end{enumerate}
\end{lemma}

This lemma was originally proved by Cotlar \cite{Cot} for  commuting self-adjoint operators, and then independently by Knapp and Stein \cite{KStein} in full generality. The proof of the Cotlar--Knapp--Stein lemma is based on identity \eqref{eq:9} or more precisely the power method in which $\|T\|_{\mathbb H \to \mathbb H}$ is expressed in terms of the operator norm of a large power of $TT^*$ or $T^*T$, which by iterating \eqref{eq:9} for any  $r\in\ZZ_+$, reads as follows
\begin{align}
\label{eq:13}
\|T\|_{\mathbb H \to \mathbb H}^{2r}=\|(TT^*)^r\|_{\mathbb H \to \mathbb H}.
\end{align}

The Cotlar--Knapp--Stein lemma is a handy tool for controlling
operators such as singular integrals or pseudo-differential operators
$T$, which can be expressed as $T =\sum_{i\in\ZZ} T_i$ of
scale-indexed pieces $T_i$ that capture the behavior at scale
$2^{-i}$. Consequently, the $L^2$-boundedness of $T$ reduces to
establishing uniform bounds of the individual $T_i$ and obtaining
sufficient decay in mixed-terms $T_i^* T_j$ and $T_i T_j^*$ when the
indices $i$ and $j$ are far apart. This can be illustrated in the
context of the Hilbert transform. For $j \in \ZZ$ and
$f \in L^2(\RR)$, we define the Hilbert transform restricted to the annulus of size $2^j$ by setting
\begin{align*}
H_j f(x) := \frac{1}{\pi} \int\limits_{2^j < |y| \leq 2^{j+1}} f(x-y) \frac{dy}{y}.
\end{align*}
The boundedness on $L^2(\RR)$ follows form the fact that each $H_jf$
is a convolution of $f$ with the kernel
$K_j(y):= \pi^{-1} y^{-1}\ind{2^j < |y| \leq 2^{j+1}}$. Moreover, it is not
difficult to see that there is a constant $C\in\RR_+$ such that
$\|K_i*K_j\|_{L^1(\RR)}\le C2^{-|i-j|}$ for any $i, j\in\ZZ$. This in
turn verifies condition \eqref{eq:cotlarsteinhyp} with $T_j=H_j$ since
$H_j^*=-H_j$ for all $j\in\ZZ$. Hence, Lemma \ref{lem:1} guarantees
that the full Hilbert transform $Hf:=\sum_{j\in\ZZ}H_jf$ is bounded on
$L^2(\RR)$. Here, as in many other situations where the
Littlewood--Paley decomposition is available, one can split the
operator $T$ into pieces $T_i$ dictated by the underlying
Littlewood--Paley decomposition, and verifying condition
\eqref{eq:cotlarsteinhyp} becomes a relatively straightforward matter.
The full strength of the Cotlar--Knapp--Stein lemma lies in settings
where Fourier methods are unavailable, for instance, on curved
manifolds or non-abelian Lie groups, or when the corresponding kernels
are rough and exhibit poor Fourier behavior.  The best known examples
falling in the latter framework are singular and maximal Radon
transforms, whose $L^p$-boundedness requires almost-orthogonality
arguments in the spirit of the Cotlar--Knapp--Stein lemma. These
operators were extensively studied by Christ, Nagel, Stein, and
Wainger in their groundbreaking paper \cite{CNSW}. Finally, we
emphasize that the Cotlar--Knapp--Stein lemma played a decisive role
in Fefferman's proof \cite{F1} of the Carleson theorem \cite{Car} on
the pointwise convergence of Fourier series for square-integrable
functions. Fefferman's ideas from \cite{F1} were further extended by
Lie \cite{Lie} to study maximally modulated polynomial Carleson
operators for all polynomial phases.

\medskip

Now, keeping in mind this brief outline of the almost-orthogonality methods  in
harmonic analysis, we can turn to the discrete setting of Radon maximal operators, which is a key aspect of the paper.

\subsection{Statement of the main result and proof strategy} Let
$\Omega$ be a convex, open subset of $\RR^d$ for which there exists a
constant $c_\Omega>0$ such that
\begin{align}
\label{eq:14}
B(0, c_\Omega)\subseteq \Omega\subseteq B(0,1),
\end{align}
and let $\Omega_k$ be its dyadic dilation defined by
\begin{equation*}
\Omega_k := 2^k\Omega = \{x \in \RR^d : x/2^k \in \Omega\}, \qquad k\in\ZZ.
\end{equation*}
Let $\Gamma\subseteq\NN^d\setminus\{0\}$ be a finite and nonempty set of multi-indices, and define a \textit{canonical polynomial mapping}  induced by the set
$\Gamma$ by setting
\begin{align*}
\RR^d\ni x\mapsto (x)^{\Gamma}:=(x^{\gamma}: \gamma\in\Gamma)\in\RR^{\Gamma}.
\end{align*}
Hence $(x)^{\Gamma}$,  is simply a vector in the $|\Gamma|$-dimensional vector space $\RR^{\Gamma}$ identified with $\RR^{|\Gamma|}$ whose entries are the monomials $x^{\gamma}:=x_1^{\gamma_1}\cdots x_d^{\gamma_d}$ with exponents $\gamma=(\gamma_1, \ldots, \gamma_d)$ being  the elements from the set $\Gamma$.

\medskip

Now having these definitions, we define  for any  function $f\colon\ZZ^{\Gamma}\rightarrow \CC$
 the \textit{averaging Radon  operator} along the canonical polynomial $\ZZ^d\ni x\mapsto (x)^{\Gamma}\in\ZZ^{\Gamma}$ by
\begin{align*}
A_k^\Gamma f(x) := \frac{1}{|\Omega_k \cap \ZZ^d|} 
\sum_{n\in\Omega_k \cap \ZZ^d}f\big(x-(n)^\Gamma\big), \qquad x \in \ZZ^{\Gamma}, \quad k\in\NN.
\end{align*}
Our main result of this paper reads as follows.
\begin{theorem}\label{thm:1.1}
Let $d\in\ZZ_+$ and a finite set
$\Gamma \subseteq \NN^{d} \setminus \{0\}$ be given. Assume that
$\Omega \subseteq \RR^d$ is a convex and  open set satisfying \eqref{eq:14}. Then, there is a constant $C>0$ such that
\begin{equation}\label{eq:1.1}
\big\|\sup_{k\in\ZZ_+}|A_k^\Gamma f|\big\|_{\ell^2(\ZZ^{\Gamma})}
\le 
C \|f\|_{\ell^2(\ZZ^{\Gamma})},\qquad f\in\ell^2(\ZZ^\Gamma).
\end{equation}
\end{theorem}

As a corollary from Theorem \ref{thm:1.1} we obtain the following maximal theorem.

\begin{theorem}\label{thm:main}
Let $d, d_0\in\ZZ_+$ be given and  $\Omega \subseteq \RR^d$ be a convex and  open set satisfying \eqref{eq:14}. Let $P=(P_1,\ldots, P_0)\colon\ZZ^d\rightarrow \ZZ^{d_0}$ be a polynomial mapping where each $P_j\colon\ZZ^d\rightarrow \ZZ$ is a  polynomial with integer coefficients.
For any  function $f\colon\ZZ^{d_0}\rightarrow \CC$ we define the averaging Radon  operator along $P$ by
\begin{align*}
A_k^P f(x) := \frac{1}{|\Omega_k \cap \ZZ^d|} 
\sum_{n\in\Omega_k \cap \ZZ^d}f\big(x-P(n)\big), \qquad x \in \ZZ^{d_0}, \quad k\in\NN.
\end{align*}
  Then, there is a constant
$C>0$ such that
\begin{equation}
\label{eq:15}
\big\|\sup_{k\in\ZZ_+}|A_k^P f|\big\|_{\ell^2(\ZZ^{d_0})}
\le 
C \|f\|_{\ell^2(\ZZ^{d_0})}, \qquad f\in\ell^2(\ZZ^{d_0}).
\end{equation}
\end{theorem} 
Theorem \ref{thm:1.1} and Theorem \ref{thm:main} are equivalent. Clearly inequality \eqref{eq:15} implies inequality \eqref{eq:14}. The reverse implication is a consequence of the lifting lemma, see \cite[Chapter~11, Section~2.4]{bigs}. Therefore, we only prove Theorem \ref{thm:1.1}. An important property is that the operators $A_k^\Gamma f$ are convolution operators.
Namely, we have $A_k^\Gamma f=f\ast K_k$, where the kernel is a discrete probability measure of the form
\begin{align}\label{eq:2.1}
	\begin{split}
K_k(x)
:=
\frac{1}{|\Omega_k \cap \ZZ^d|}\sum_{n\in\Omega_k \cap \ZZ^d}
\ind{\{(n)^\Gamma\}}(x) =\int_{\TT^\Gamma} \ex(x\cdot\xi)S_k(\xi)\,d\xi,
\end{split}
\end{align}
with
\begin{equation*}
S_k(\xi)
:=
\frac{1}{|\Omega_k \cap \ZZ^d|}\sum_{n\in \Omega_k \cap \ZZ^d} \ex(-(n)^\Gamma\cdot\xi).
\end{equation*}

Theorem \ref{thm:main} for $d=1$ was proved by Bourgain in a series
of groundbreaking papers \cite{B1,B2,B3}. The case of general
$d \in \mathbb{Z}_+$ was understood in \cite{MST1}. In these papers it
was also shown that these $\ell^2$ bounds extend to the corresponding
$\ell^p$ bounds for all $p \in (1,\infty]$. Here we have chosen to
prove inequality \eqref{eq:1.1} as an illustration of
almost-orthogonality methods in the discrete analogues of harmonic
analysis, and we do not aspire to claim that we proved anything
new.

These days the inequality \eqref{eq:1.1} is very well
understood and has a fairly clean proof even on $\ell^p(\ZZ^\Gamma)$ spaces for
$p \in (1,\infty]$. Namely, exploiting the convolution structure combined with Plancherel's theorem and the Hardy--Littlewood circle method, see for instance \cite{MST1, MST2, advances, MSS}, one can construct an approximating operator for \(A_k^\Gamma f\), which on the Fourier-transform side will look like \(S_k(\xi)\) restricted to major arcs, i.e. a family of intervals centered at reduced rational fractions with small denominators. The symbol \(S_k(\xi)\) restricted to the minor arcs, i.e. the complement of the major arcs, is highly oscillatory and its contribution can be handled by Plancherel’s theorem combined with Weyl’s inequality from Lemma \ref{lem:A3}. On major arcs the approximating operator can be split into small and large scales. Then the argument can be completed by appealing to the Rademacher--Menshov inequality (see inequality \eqref{eq:16}) and the Ionescu--Wainger multiplier theorem \cite{IW} for small scales and the sampling method of Magyar--Stein--Wainger \cite{MSW} for large scales, where the bounds are controlled by their continuous counterparts.

Now in our proof of Theorem \ref{thm:1.1} we will also rely on using the Hardy--Littlewood circle method, but instead of using classical Fourier analysis based on Plancherel's theorem we will use the $TT^*$ methods. Our motivation for presenting the proof of Theorem \ref{thm:1.1} using almost-orthogonality methods emerged from the following two remarks:

\begin{itemize}
\item[(i)] Almost-orthogonality methods were developed in the field of
discrete analogues of harmonic analysis from its very beginning
\cite{IMSW, KR, MSWh, PPhd, SWd}. However, these methods were only
applied to very specific classes of Radon operators --- specifically,
Radon operators along polynomial mappings, which imposed severe
restrictions on the degree of the polynomial mappings. A significant
progress was made in \cite{IMW}, where the authors established
$\ell^2$-boundedness of the discrete singular integral Radon transform
on nilpotent groups of step two for arbitrary polynomial
mappings. This was a remarkable breakthrough in the field that shed
light on other discrete analogues in harmonic analysis as well as
other problems having a discrete flavor.

\smallskip
\item[(ii)] Only recently, almost-orthogonality methods were developed
in \cite{IMMS} to prove an analogue of Theorem \ref{thm:1.1} on
discrete nilpotent groups of step two. The ideas invented in
\cite{IMMS} resulted in a robust \textit{nilpotent circle method} that
allowed the authors to handle nilpotent variants of maximal
inequalities from \eqref{eq:1.1} as well as variants of Waring
problems on universal nilpotent groups of step two \cite{IMMSWar}. The
proof methods from \cite{IMMS, IMMSWar} developed almost-orthogonality
techniques that go far beyond classical Fourier tools, which are not
available in the non-commutative nilpotent setting. We believe that
the methods from \cite{IMMS, IMMSWar} have prospects to make progress
on similar problems on nilpotent groups of higher order and contribute
to the Bergelson--Leibman conjecture \cite{BL}, a major open problem in
pointwise ergodic theory, see  also to \cite{IMMS,
KMPWW}.
\end{itemize}

In view of these two remarks, it is clear that almost-orthogonality
methods play a key role in the field of discrete analogues of harmonic
analysis, and Theorem \ref{thm:1.1} provides a simple model to
illustrate the key features of the $TT^*$ methods. We also believe that
presenting the key ideas from \cite{IMMS} in the context of inequality
\eqref{eq:1.1} may be a starting point for understanding discrete
variants of maximally modulated polynomial Carleson--Radon type
operators from \cite{AMPY, GPRY,  PY}. This also appears to be an
interesting new direction for investigations in the field of discrete analogues of harmonic
analysis.

As mentioned above, in our argument the Plancherel theorem is replaced
by $TT^*$ arguments. In fact, we will use two types of
almost-orthogonality arguments: explicitly, identity \eqref{eq:9} to
handle all pieces of the kernel from \eqref{eq:2.1} that contribute to
the error term and produce decay in scales, and the
Cotlar--Knapp--Stein lemma to handle all other terms contributing to
the asymptotic formula of our kernel \eqref{eq:2.1}.
However, working with Radon-type operators the
$TT^*$ arguments cannot be applied directly; instead, as in
\cite{CNSW} and \cite{IMW, IMMS}, we have to consider high-order
$TT^*$ arguments as in \eqref{eq:13}, i.e. \((TT^*)^r\) for
 large \(r \in \mathbb{Z}_+\), which exhibits more
regularity. This can be easily  seen for instance using $M_t$ from
\eqref{eq:17}. Indeed, for any \(r \in \mathbb{Z}_+\),  the
multiplier corresponding to \((M_tM_t^*)^r\) satisfies
\(\big|\int_0^1 \ex(t\xi y)\, dy\big|^{2r} \lesssim (1+|t\xi|)^{-2r}\),
while the multiplier corresponding to \(M_t\) satisfies
\(\big|\int_0^1 \ex(t\xi y)\, dy\big| \lesssim (1+|t\xi|)^{-1}\). Hence
the operator \((M_tM_t^*)^r\) becomes more regular as \(r\)
increases. This approach was exploited in the continuous setting in
\cite{CNSW}, and later in \cite{IMW, IMMS}, to bypass using
Plancherel's theorem, which is not available in non-commutative
settings.

Since the Hardy--Littlewood circle method also underpins our
argument and dictates how to decompose our kernel \eqref{eq:2.1},
therefore studying  \((TT^*)^r\) for large
\(r \in \mathbb{Z}_+\) corresponds to a simple arithmetic 
 heuristic lying behind the proofs of Waring-type problems. This heuristic says that, the more variables
that occur in the Waring-type equation, the easier it is to find a
solution. By adjusting the parameter \(r\) (usually taking \(r\) very
large), we can always decide how many variables we have at our
disposal, making the operators in our questions ``smoother and
smoother''. For $(M_t M_t^*)^r$ it is manifested how fast the
multiplier decays at infinity as we have seen above. We now focus on proving Theorem \ref{thm:1.1}.

\section{Notation and basic tools}
\label{sec:not}

We now set up the notation and gather the basic tools that will be used throughout the article. 

\subsection{Basic notation}
Let $\ZZ_+ := \{1,2,\dots\}$ be the set of positive integers,
$\NN := \{0,1,2,\dots\}$ be the set of nonnegative integers,
$\RR_+ := (0,\infty)$ be the set of positive real numbers, and
$\TT$ be the one-dimensional torus. Given
$m \in \Z_+$, the sets $\ZZ^m, \, \QQ^m, \, \RR^m$, $ \CC^m$ and
$\TT^m$ have standard meaning, and 
$\ZZ_Q^m :=\{0,1, \dots, Q-1\}^m$ for any $Q\in\ZZ_+$.

For
any $x\in\RR$ we define the floor function
$\lfloor x \rfloor: = \max\{ n \in \ZZ : n \le x \}$ and 
the fractional part 
$\{x\}:=x-\lfloor x\rfloor$ as well as the distance to the nearest integer
by $\|x\|:={\rm dist}(x, \ZZ)$. 

For
any positive real number $N\in\RR_+$ we define
\[
[N]:=\ZZ_+\cap(0, N]=\{1, 2,\ldots, \lfloor N \rfloor\}.
\]

For $a = (a_1,\ldots, a_m) \in \ZZ^m$ and $q\in\ZZ_+$, we denote by ${\rm gcd}(a,q):={\rm gcd}(a_1,\ldots, a_m,q)$ the greatest
common divisor of $a$ and $q$; that is, the largest  $n \in\ZZ_+$ that divides $q$ and all the components
$a_1, \ldots, a_m$. Clearly, any vector in $\QQ^m$ has a unique representation as $a/q$ with $q\in \ZZ_{+}$,
$a \in \ZZ^m$, and ${\rm gcd}(a,q)=1$.

If $A$ is a set, its indicator function will be denoted by
$\ind{A}$. When $A$ is finite, its number of elements is denoted by
$|A|$. If $S$ is a statement, then $\ind{S}$ denotes its indicator,
equal to $1$ if $S$ is true and $0$ if $S$ is false.  In particular,
we have $\ind{x\in A} \equiv \ind{A}(x)$.

Finally, for a finite set $A$, we will write
$\mathbb X^A:=\mathbb X^{|A|}$, whenever
$\mathbb X\in\{\ZZ, \QQ, \RR, \CC, \TT\}$.  This notation will be
mainly applied to \(A=\Gamma\), where
\(\Gamma \subseteq \mathbb{N}^d \setminus \{0\}\) is a finite set of multi-indices. Then an element
\(x \in \mathbb{X}^\Gamma\) is a vector of the form
$x := (x_\gamma)_{\gamma \in \Gamma}=(x_\gamma:\gamma \in \Gamma)$,
whose entries \(x_\gamma\) are elements of \(\mathbb{X}\).

\subsection{Asymptotic notation}
For two nonnegative quantities $A, B$ we write $A \lesssim B$ if there
is an absolute constant $C>0$, which may change from line to line,
such that $A\le CB$. We also write $A \simeq B$ when
$A \lesssim B\lesssim A$. We will write $\lesssim_{\delta}$ or
$\simeq_{\delta}$ to emphasize that the implicit constant depends on
$\delta$.

For two functions $f \colon X\to \C$ and  $g \colon X\to \RR_+$, we write $f = O(g)$ if there
exists a constant $C>0$ such that $|f(x)| \le C g(x)$ for all
$x\in X$. We will also write $f = O_{\delta}(g)$ if the implicit
constant depends on $\delta$.

Throughout the paper, for any $A\in \RR$, we will use the following trivial estimate
\begin{equation}\label{eq:dee}
1+|A + O(1)|\simeq 1+|A|,
\end{equation}
where the implied constant does not depend on $A$.

\subsection{Euclidean spaces}
The standard inner product on $\RR^m$ is denoted by
$x\cdot\xi:=\sum_{k\in [m]}x_k\xi_k$ for every $x=(x_1,\ldots, x_m)$
and $\xi=(\xi_1, \ldots, \xi_m)\in\RR^m$. The inner product induces
the Euclidean norm $|x|_2:=\sqrt{x\cdot x}$, which will be abbreviated
to $|x|$. For any $x\in\RR^m$ we will also consider its $\ell^\infty$-norm given by
$|x|_{\infty} := \max\{\abs{x_j} : j\in[m]\}$. Clearly we have $|x|_{\infty}\le |x|_{2}\le \sqrt{m}|x|_{\infty}$ for any $x\in\RR^m$.

Throughout the paper the $m$-dimensional torus $\TT^m = \RR^m/\ZZ^m$, which unless otherwise stated will be  identified with
$[-1/2, 1/2)^m$, is a priori endowed with the   periodic norm
\begin{align}
	\label{eq:99}
	\|\xi\|:=\Big(\sum_{k\in[m]} \|\xi_k\|^2\Big)^{1/2}
	\qquad \text{for}\qquad
	\xi=(\xi_1,\ldots,\xi_m)\in\TT^m,
\end{align}
where $\|\xi_k\|=\dist(\xi_k, \ZZ)$ for all $\xi_k\in\TT$ and
$k\in[m]$. However, identifying $\TT^m$ with $[-1/2, 1/2)^m$, we
see that the norm $\|\cdot\|$ coincides with the Euclidean norm
$\abs{\:\cdot\:}$ restricted to $[-1/2, 1/2)^m$.

\subsection{Function spaces} All vector spaces in this paper will be
defined over the complex numbers $\CC$. The triple
$(X, \mathcal B(X), \mu)$ is a measure space $X$ with $\sigma$-algebra
$\mathcal B(X)$ and $\sigma$-finite measure $\mu$.  The space of all
measurable functions whose modulus is integrable with $p$-th power is
denoted by $L^p(X)$ for $p\in(0, \infty)$, whereas $L^{\infty}(X)$
denotes the space of all essentially bounded measurable functions.

Here, we will 
have $X=\RR^m$ or $X=\TT^m$ equipped with the Lebesgue measure, and
$X=\ZZ^m$ endowed with the counting measure. If $X$ is endowed with
counting measure we will abbreviate $L^p(X)$  to
$\ell^p(X)$. 

For $T \colon B_1 \to B_2$, a continuous linear (sublinear)  map between two normed
vector spaces $B_1$ and  $B_2$, we use $\|T\|_{B_1 \to B_2}$ to denote its
operator norm. This will be usually used with $B_1=B_2=\ell^p(\ZZ^m)$.

\subsection{Smooth functions and derivatives}
Let ${\bm \eta} \colon \RR\to[0, 1]$ be a smooth and even cutoff function
such that $\ind{[-1, 1]}\le{\bm \eta}\le \ind{(-2, 2)}$. Additionally, one can  think that ${\bm \eta}$ is nondecreasing on $(-\infty, 0]$ and nonincreasing on $[0, \infty)$.   For a given
$\Gamma\subseteq\NN^d \setminus \{0\}$, let $\eta_0 \colon \RR^\Gamma\rightarrow[0,1]$
be given by $\eta_0(x):=\prod_{\gamma\in\Gamma}{\bm \eta}(x_{\gamma})$ for any $x=(x_\gamma:\gamma\in \Gamma)\in\RR^\Gamma$. 
From this definition we see that the function $\eta_0$ is smooth, even, and has a product structure, moreover it satisfies
$\eta_0(x)=1$ for $x\in[-1,1]^\Gamma$ and  $\supp\eta_0\subseteq[-2,2]^\Gamma$. 
Next define
\begin{equation}
\label{eq:1}
\eta_{\le A}(x) :=\eta_0(2^{-A}x), \qquad A>0.
\end{equation}

The partial derivative of a function $f \colon \RR^m\to\CC$ with respect to
the $j$-th variable $x_j$ will be denoted by
$\partial_{x_j}f=\partial_j f$, while the $k$-fold partial derivative
will be denoted by
$\partial_{x_j}^kf=\partial_j^k f$ with the convention that $\partial_j^0f=f$ for any $j\in[m]$. For any multi-index $\alpha=(\alpha_1, \ldots, \alpha_m)\in\mathbb N^m$, we define 
\begin{align*}
D^{\alpha}f:=\partial_1^{\alpha_1}\cdots \partial_m^{\alpha_m}f.
\end{align*}

For functions  $f \colon \RR^\Gamma\to\mathbb C$ with  a finite set of multi-indices $\Gamma\subseteq \NN^d$, we will write $\partial_\gamma^k f $ to denote the partial derivative of $f$ of order $k$ with respect to a variable in $\RR^\Gamma$ with the index $\gamma\in\Gamma$.

\subsection{Fourier transform and Fourier series}  
Recall that $\ex (z)=e^{\dpi z}$  for every $z\in\CC$.
The Fourier transform and the inverse Fourier transform of $f\in L^1(\RR^m)$ will be denoted respectively by
\begin{align*}
\mathcal F_{\RR^m} f(\xi) &:= \int_{\RR^m} f(x) \ex (-x\cdot \xi)\ dx,\qquad \xi\in\RR^m,\\
 \mathcal F_{\RR^m}^{-1} f(x) &:= \int_{\RR^m} f(\xi) \ex (x\cdot \xi)\ d\xi,\qquad x\in\RR^m.
\end{align*}

\subsection{Canonical polynomial mappings}

For $\gamma=(\gamma_1,\ldots, \gamma_d)\in\NN^d$ and $x=(x_1,\ldots, x_d)\in\RR^d$ define 
\begin{align}
\label{eq:2}
x^{\gamma}:= x_1^{\gamma_1}\cdots\:  x_d^{\gamma_d}.
\end{align}
In other words, $x^{\gamma}$ is a  monomial of degree $\gamma_1+\ldots+ \gamma_d$ with variables $x_1,\ldots, x_d$ and exponents $\gamma_1,\ldots, \gamma_d$. For any multi-index $\gamma=(\gamma_1,\dots,\gamma_d)\in\N^d$, by a slight abuse of notation we write
$|\gamma|:=\gamma_1+\ldots+\gamma_d$ for the length of $\gamma$. This will never cause confusions
since the multi-indices will be always denoted by Greek letters.

Fix a finite and nonempty set  of multi-indices $\Gamma\subseteq\NN^d\setminus\{0\}$. Recall that a \textit{canonical polynomial mapping}  induced by the set
$\Gamma$ is given by
\begin{align}
\label{eq:3}
\RR^d\ni x\mapsto (x)^{\Gamma}:=(x^{\gamma}: \gamma\in\Gamma)\in\RR^{\Gamma}.
\end{align}
Hence $(x)^{\Gamma}$,  is simply a vector in the $|\Gamma|$-dimensional vector space whose entries are the monomials $x^{\gamma}$ indexed by the set $\Gamma$. An important feature of the polynomials in \eqref{eq:3} is their scale homogeneity.

Namely, for any $\Lambda>0$ and $x=(x_{\gamma}: \gamma\in\Gamma)\in\RR^\Gamma$, we define
$\Lambda\circ x \in \RR^\Gamma$ by
\begin{equation*}
(\Lambda\circ x)_\gamma := \Lambda^{|\gamma|}x_\gamma, \qquad \gamma\in\Gamma.
\end{equation*}
Notice that $\RR^\Gamma\ni x\mapsto \Lambda\circ x\in \RR^\Gamma$ is a homomorphism on $\RR^\Gamma$, i.e.\ 
\begin{align*}
\Lambda\circ (x+y) = \Lambda\circ x + \Lambda\circ y, \qquad x,y \in  \RR^\Gamma.
\end{align*}
Further, 
the above dilation operation leads to the following identity
\begin{equation*}
\Lambda\circ \big((x)^\Gamma\big)=(\Lambda x)^\Gamma.
\end{equation*}
Note that for any $x,y\in\RR^\Gamma$ we have
\begin{equation*}
(\Lambda\circ x)\cdot y=x\cdot(\Lambda\circ y).
\end{equation*}
We will frequently use this fact without mentioning it.

\subsection{Approximation sums by integrals}
The following lemma will be useful.
\begin{lemma}\label{lem:A1}
Fix $d\in\ZZ_+$ and a finite nonempty set of multi-indices
$\Gamma \subseteq \NN^{d} \setminus \{0\}$.
Let $M\ge |\Gamma|+1$ and $f \colon \RR^\Gamma\rightarrow\CC$ be a Schwartz function. Then for any $\xi\in[-1/2, 1/2]^\Gamma$ we have
\begin{align*}
\Big|\sum_{n\in\ZZ^\Gamma} f(n) \ex(n\cdot \xi)-\int_{\RR^\Gamma} f(x) \ex(x\cdot \xi)\,dx\Big|\lesssim_{M, \Gamma}\int_{\RR^\Gamma}\sum_{\gamma\in\Gamma} |\partial_\gamma^M f(x)|\,dx.
\end{align*}
Consequently, for every $\gamma_0\in\Gamma$ we have
\begin{align*}
\Big|\sum_{n\in\ZZ^\Gamma} f(n) \ex(n\cdot \xi)\Big|\lesssim_{M, \Gamma} |\xi_{\gamma_0}|^{-M}\int_{\RR^\Gamma} |\partial_{\gamma_0}^M f(x)|\,dx+\int_{\RR^\Gamma}\sum_{\gamma\in\Gamma} |\partial_\gamma^M f(x)|\,dx.
\end{align*}
\end{lemma}

\begin{proof}
The second part of the claim is a simple consequence of the first part and the integration by parts applied to the integral $\int_{\RR^\Gamma} f(x) \ex(x\cdot \xi)\,dx$. Thus it suffices to show the first estimate.

By the Poisson summation formula for any $\xi\in[-1/2, 1/2]^\Gamma$ we have
\begin{align*}
\sum_{n\in\ZZ^\Gamma} f(n) \ex(n\cdot \xi)
=\sum_{n\in\ZZ^\Gamma} \int_{\RR^\Gamma}f(x) \ex\big(x\cdot (\xi-n)\big)\, dx=\int_{\RR^\Gamma}f(x)\ex\big(x\cdot \xi\big)\,dx+\sum\limits_{n\in\ZZ^\Gamma \setminus \{0\} }E_n(\xi),
\end{align*}
where
$$
E_n(\xi):= \int_{\RR^\Gamma}f(x) \ex \big(x\cdot (\xi-n)\big)\, dx,
\qquad n\in\ZZ^\Gamma \setminus \{0\}.
$$
Since for $M\ge |\Gamma|+1$ we have
$$
\sum\limits_{ n\in\ZZ^\Gamma \setminus \{0\} } |\xi-n|^{-M}\lesssim 1,
$$
uniformly in $\xi\in[-1/2,1/2]^\Gamma$, to finish the proof it suffices to show that 
\begin{align*}
|E_n(\xi)|\lesssim |\xi-n|^{-M}\int_{\RR^\Gamma}\sum_{\gamma\in\Gamma}|\partial_\gamma^M f(x)|\,dx.
\end{align*}
To this end for fixed $\xi$ and $n$ let $\gamma_0\in\Gamma$ be such that $|\xi_{\gamma_0}-n_{\gamma_0}|=\max_{\gamma\in\Gamma} |\xi_\gamma-n_\gamma|\simeq |\xi-n|$.
Then integrating by parts $M$ times with respect to $x_{\gamma_0}$ we obtain
\begin{align*}
|E_n(\xi)|
&= (2 \pi)^{-M} |\xi_{\gamma_0}-n_{\gamma_0}|^{-M} \Big|\int_{\RR^\Gamma}f(x)\partial_{\gamma_0}^M\big(\ex(x\cdot(\xi-n))\big)\,dx\Big|
\lesssim 
|\xi-n|^{-M}\int_{\RR^\Gamma}\sum_{\gamma\in\Gamma} |\partial_\gamma^M f(x)|\,dx.
\end{align*}
That concludes the proof of the lemma.
\end{proof}

\subsection{Weyl's inequality} A multidimensional Weyl inequality reads as follows.
\begin{lemma} [{\cite[Theorem A.1, p.\ 49]{advances}}] \label{lem:A3}
Fix $d\in\ZZ_+$ and a finite nonempty set of multi-indices
$\Gamma \subseteq \NN^{d} \setminus \{0\}$. There exists $\epsilon>0$ depending only on $d$ and $\Gamma$ such that, for every polynomial
$P(n) = \sum_{\gamma\in\Gamma} \xi_{\gamma} n^{\gamma}$, every $N>1$, convex set $K \subseteq B(0,N) \subseteq \R^{d}$, multi-index $\gamma_{0}\in\Gamma$, and integers $a \in \ZZ$ and $q \in\ZZ_+$ with $\gcd(a, q) = 1$ and $|\xi_{\gamma_0}-a/q|\le q^{-2}$,
we have
\begin{equation}
\label{eq:56}
N^{-d}\Big|\sum_{n \in K \cap \Z^d} e(P(n))\Big|
\lesssim_{d,\Gamma} \kappa^{-\epsilon} \log (N+1),
\end{equation}
where
\[
\kappa := \min \{q,N^{\abs{\gamma_{0}}}/q\}.
\]
The implicit constant in \eqref{eq:56} is independent of the coefficients of  $P$ and the numbers $a$, $q$, $N$.
\end{lemma}

\subsection{Complete exponential sums}

For $q\in\ZZ_+$ and $a\in \ZZ^\Gamma$ satisfying $\gcd(a,q)=1$ define a  normalized complete exponential sums, i.e. the  Gaussian sum
\begin{align} \label{def:Gauss}
G(a/q)=q^{-|\Gamma|}\sum_{r\in[q]^\Gamma} \ex(-(r)^{\Gamma}\cdot a/q),
\end{align}
where $(r)^\Gamma$ is a canonical polynomial mapping defined by \eqref{eq:3}.
The following estimate, which is a consequence of Lemma~\ref{lem:A3}, will be used frequently.

\begin{lemma}\label{lem:A4}
Fix $d\in\ZZ_+$ and a finite nonempty set of multi-indices
$\Gamma \subseteq \NN^{d} \setminus \{0\}$.
There exists $\epsilon'>0$ depending only on $d$ and $\Gamma$ such that
\begin{align*}
|G(a/q)|\lesssim_{d,\Gamma} q^{-\epsilon'}, \qquad a/q \in \QQ^d,
\end{align*}
provided that $\gcd(a,q)=1$.
\end{lemma}
\subsection{Oscillatory integrals estimates}
We shall also use the following van der Corput lemma.
\begin{lemma} [{\cite[Proposition~2.1]{SW}}] \label{lem:A2}
Fix $d\in\ZZ_+$ and a finite nonempty set of multi-indices
$\Gamma \subseteq \NN^{d} \setminus \{0\}$.	Let $P(x)=\sum_{\gamma\in\Gamma} \lambda_\gamma x^\gamma$ be a polynomial in $\RR^d$  with real coefficients of degree $\le d_{\Gamma} := \max\{|\gamma| : \gamma \in \Gamma\}$.  Then for any convex subset $\Omega$ of the closed unit ball $B$ in $\RR^d$, and for any $C^1$ function $\varphi$ on $B$, we have
	\begin{align*}
		\Big|\int_\Omega \ex(P(x))\varphi(x)\,dx\Big|
		\lesssim_{d,\Gamma} 
		\big(\max_{\gamma\in\Gamma} |\lambda_\gamma| \big)^{-1/d_{\Gamma}} \sup_{x\in B}(|\varphi(x)|+|\nabla \varphi(x)|).
	\end{align*}
	
\end{lemma}

\subsection{Choice of parameters}
We now fix $d\in\ZZ_+$ and a finite nonempty set of multi-indices
$\Gamma \subseteq \NN^{d} \setminus \{0\}$.
For future reference and for notational convenience, we will apply the
above three lemmas in such a way that the implied power savings have
the same exponent. Therefore, we define
\begin{equation}\label{eq:tau}
\tau:=\min\{\epsilon, \epsilon', 1/d_{\Gamma}, 1/|\Gamma| \},
\end{equation}
where $\epsilon$ and $\epsilon'$ are constant from Lemmas \ref{lem:A3} and \ref{lem:A4}, and $d_{\Gamma}$ is the quantity from Lemma \ref{lem:A2}.

We further define the following quantities:
\begin{enumerate}[label*={(\alph*)}, itemsep=2pt ]
\item \label{item1} Let $\delta:=\delta(d,\Gamma)\in(0, 1)$ be a small positive constant satisfying $\delta<\big(100(d+|\Gamma| + d_{\Gamma})\big)^{-100}$. 
\item \label{item2} Let $r:=r(d,\Gamma, \delta)$ be a large dyadic constant such that $r\tau\ge \delta^{-4}$, where $\tau$ is defined in \eqref{eq:tau}.
\item \label{item3} Let $D:=D(d,\Gamma, \delta, r)\in\RR_+$ be such that $D>r\delta^{-4}$.
\item  \label{item4} Finally, for $s\ge 0$, we define
\begin{align}
\label{eq:4}
\kappa_s :=2^{2D(s+1)^2}.
\end{align}
\end{enumerate}

\subsection{Auxiliary lemma}
We introduce the family of  cutoff functions
\begin{equation}
\label{eq:5}
\phi_k(x):= \phi_{k,\delta}(x) := \eta_{\le \delta k}(2^{-k}\circ x), 
\qquad x \in  \RR^\Gamma, \quad k \in \NN.
\end{equation}
Further, we introduce the following auxiliary periodic function
\begin{equation}\label{eq:6.4}
	\Phi_k(\beta):=\sum_{g\in\ZZ^\Gamma} \phi_k(g) \ex(-g\cdot\beta), 
	\qquad \beta\in \TT^\Gamma, \quad k \ge D.
\end{equation}
We have the following useful estimate.
\begin{lemma}\label{lem:est}
	For a fixed $M \ge |\Gamma| + 1$
	we have 
\begin{equation}\label{eq:6.5}
	|\Phi_k(\beta)|
	\lesssim_{M}
	\prod_{\gamma\in\Gamma}2^{k(|\gamma|+\delta)} 
	\big(1+2^{k(|\gamma|+\delta)} \|\beta_\gamma\| \big)^{-M}, 
	\qquad \beta\in \TT^\Gamma, \quad k \ge D,
\end{equation}
and in particular
\begin{equation}\label{eq:FkL1}
	\|\Phi_k\|_{L^1(\TT^\Gamma)}\lesssim 1, \qquad k \ge D.
\end{equation}
\end{lemma}

\begin{proof}
The estimate \eqref{eq:FkL1} is a simple consequence of
\eqref{eq:6.5}, so we prove only the latter estimate.  Without any loss of
generality we may restrict our attention to
$\beta \in [-1/2,1/2]^{\Gamma}$.  By Lemma \ref{lem:A1} we have for
any $M\in\NN$ with $M \ge |\Gamma| + 1$ the identity
\begin{equation*}
	\Phi_k (\beta)
	=\int_{\R^\Gamma} \phi_k(x)\ex(- x\cdot\beta)\,dx
	+
	O\Big(\int_{\R^\Gamma}\sum_{\gamma\in\Gamma}|\partial_\gamma^M \phi_k(x)|\,dx\Big).
\end{equation*}
Clearly if $M$ is sufficiently large we have
\begin{equation*}
	\int_{\R^\Gamma}\sum_{\gamma\in\Gamma}|\partial_\gamma^M \phi_k(x)|\,dx\lesssim 2^{-Mk}|\supp \phi_k|\lesssim 2^{-Mk/2},
\end{equation*}
so it remains to treat $\int_{\R^\Gamma} \phi_k(x)\ex(-x\cdot\beta)\,dx$. Indeed, integrating by parts with $\alpha=(\alpha_\gamma)_{\gamma\in\Gamma}$, where $\alpha_\gamma= M$ if $|\beta_\gamma|>2^{-k(|\gamma|+\delta)}$ and  $\alpha_\gamma=0$ if $|\beta_\gamma|\le 2^{-k(|\gamma|+\delta)}$, we obtain
\begin{align*}
	\Big|\int_{\R^\Gamma} \phi_k(x)\ex(x\cdot\beta)\,dx\Big|
	& \simeq
	\Big|\int_{\R^\Gamma} \phi_k(x) D_{x}^\alpha\Big(\frac{\ex(x\cdot\beta)}{\beta^\alpha}\Big)\,dx\Big|\le |\beta^{-\alpha}|\int_{\R^\Gamma} |D^\alpha\phi_k(x)|\,dx
	\\
	&\lesssim 
	|\supp\phi_k|
	\prod_{\gamma\in\Gamma}\big(1+2^{k(|\gamma|+\delta)}|\beta_\gamma|\big)^{-M}
	\simeq
	 \prod_{\gamma\in\Gamma}2^{k(|\gamma|+\delta)}  \big(1+2^{k(|\gamma|+\delta)}|\beta_\gamma|\big)^{-M},
\end{align*}
which proves \eqref{eq:6.5} as desired.
\end{proof}

\subsection{General concept behind the high-order $TT^*$ argument}\label{subsec:TT*}
Assume that $T\colon \ell^2(\ZZ^\Gamma)\rightarrow \ell^2(\ZZ^\Gamma)$
is a linear and bounded operator.  The $TT^*$ method is based on the observation
that for any  $r\in\ZZ_+$, we have
\begin{align*}
\|T\|_{\ell^2(\ZZ^\Gamma)\rightarrow \ell^2(\ZZ^\Gamma)}=\|(T^* T)^r\|^{1/2r}_{\ell^2(\ZZ^\Gamma)\rightarrow \ell^2(\ZZ^\Gamma)}.
\end{align*} 
In general, as in \cite{CNSW} and as observed above, the expectation
is that the operator $(T^* T)^r$ should exhibit more regularity and
consequently be easier to analyze than $T$. We aim to apply this method
to a particular case in which $T$ is  a convolution operator. Our analysis relies on a concise representation of
the convolution kernel of $(T^* T)^r$, which we derive below. We
present the argument in a slightly more general setting, where instead
of a single operator we consider a family of operators.

Assume that we have convolution operators $U_1, T_1, \dots, U_r, T_r \colon \ell^2(\ZZ^\Gamma)\rightarrow \ell^2(\ZZ^\Gamma)$, with the $\ell^1(\ZZ^\Gamma)$ kernels $L_1, K_1, \dots, L_r, K_r$, respectively, i.e.
\begin{align*}
U_j f=f\ast L_j, \qquad T_j f=f\ast K_j, \qquad 1\le j \le r.
\end{align*}
Then the adjoint operators $U_1^*, \dots, U_r^*$ are also convolution operators with kernels $L_j^*(x)=\overline{L_j(-x)}$, i.e.
\begin{equation*}
U_j^* f=f\ast L_j^*.
\end{equation*}
For $f\in\ell^2(\ZZ^\Gamma)$ and $x\in\ZZ^\Gamma$ we have 
\begin{align*}
U_1^*T_1\dots U_r^* T_r f(x)=f\ast A^r (x):=\sum_{h_1, g_1, \dots, h_r, g_r\in\ZZ^\Gamma}\Big(\prod_{j=1}^r L_j^*(h_j) K_j(g_j)\Big) f\Big(x-\sum_{j=1}^r (g_j+h_j)\Big),
\end{align*}
where
\begin{equation*}
A^r(y):=\sum_{h_1, g_1, \dots, h_r, g_r\in\ZZ^\Gamma}\Big(\prod_{j=1}^r \overline{L_j(h_j)} K_j(g_j)\Big) \ind{\{0\}} \Big(y-\sum_{j=1}^r (g_j - h_j) \Big).
\end{equation*}
Using the identity
\begin{equation*}
\ind{\{0\}} (z)=\int_{\TT^\Gamma} \ex(z\cdot\theta)\,d\theta, \qquad z \in \ZZ^{\Gamma},
\end{equation*}
we obtain
\begin{align*}
A^r(y)=\int_{\TT^\Gamma}\ex(y\cdot\theta)\Big[\prod_{j=1}^r\Big(\sum_{h_j\in\ZZ^\Gamma}\overline{L_j(h_j)} \ex(h_j\cdot\theta)\Big)\Big( \sum_{g_j\in\ZZ^\Gamma}K_j(g_j) \ex(-g_j\cdot\theta)\Big)\Big]\, d\theta.
\end{align*}
We will refer to these calculations further in the paper. 

\section{Proof of Theorem \ref{thm:1.1}}
In this section, we prove Theorem \ref{thm:1.1}. We begin with
preliminaries originating from the circle method of Hardy and
Littlewood, which will set up the proof strategy. There is a vast
literature in analytic number theory on the classical circle method;
the books by Iwaniec and Kowalski \cite{IK}, Nathanson \cite{Nat},
Vaughan \cite{Vau}, Vinogradov \cite{Vin}, and the notes of Wooley
\cite{Wooleynotes} are excellent references on this  subject.

\medskip

For $s\in \NN$, we define the set of reduced fractions with denominators of size roughly $2^s$ by
\begin{equation*}
R_s^\Gamma:=\big\{a/q\in \QQ^\Gamma : a\in\ZZ^\Gamma, \quad q\in[2^s, 2^{s+1}-1]\cap \ZZ, \quad \gcd(q, \gcd(a_\gamma:\gamma\in\Gamma))=1 \big\}.
\end{equation*}
Moreover, for each $t>0$ we let $R^\Gamma_{\le t}:=\bigcup_{0 \le s\le t}R_s^\Gamma$. 

Now, for $k\ge D$ and $s\le \delta k$ we define the periodic Fourier multipliers $\Xi_{k,s}\colon \RR^\Gamma\rightarrow[0,1]$ as follows 
\begin{align*}
\Xi_{k,s}(\xi) & := \sum_{a/q \in R_s^\Gamma} \eta_{\le \delta k}\big(2^k\circ (\xi-a/q)\big), \qquad \xi\in \RR^\Gamma, \\
\Xi_{k}^c(\xi)
& :=
 1-\sum_{0\le s\le \delta k}\Xi_{k,s}(\xi), \qquad \xi\in \RR^\Gamma.
 \end{align*}
These definitions are motivated by the circle method of Hardy and Littlewood. In the language of this theory, the function $\Xi_{k,s}$ is supported on the so-called \textit{major arcs} --- the union of intervals centered at fractions from \(R_s^\Gamma\) --- and \(\Xi_{k}^c\) is supported on the complementary set, called the \textit{minor arcs}.

\medskip

Using \eqref{eq:2.1}, \eqref{eq:5} and the fact that $\supp K_k \subseteq \{x \in \RR^{\Gamma} : |2^{-k} \circ x| \lesssim 1\}$, for $k \ge D$ we can write
\begin{equation}
\label{eq:6}
K_k(x)=K_k^c(x)+\sum_{0\le s\le\delta k} K_{k,s}(x),
\end{equation}
where 
\begin{align} \label{id:604}
	\begin{split}
K_k^c(x) &:= \phi_k(x)\int_{\TT^\Gamma} \ex(x\cdot \xi) S_k(\xi) \Xi_{k}^c(\xi)\,d\xi, \\
K_{k,s}(x) &:=\phi_k(x)\int_{\TT^\Gamma} \ex (x\cdot \xi) S_k(\xi) \Xi_{k,s}(\xi)\,d\xi.
\end{split}
\end{align}

In view of decomposition \eqref{eq:6} and the definition of $\kappa_s$ from \eqref{eq:4}, the proof of Theorem \ref{thm:1.1} reduces to showing the following three lemmas.
\begin{lemma}\label{lem:minor arcs}
	We have 
\begin{align*}
\|f\ast K_k^c\|_{\ell^2(\ZZ^\Gamma)}\lesssim 2^{-k/D}\|f\|_{\ell^2(\ZZ^\Gamma)},
\end{align*}
uniformly in $k\ge D$ and $f\in\ell^2(\ZZ^\Gamma)$.
\end{lemma}

\begin{lemma}\label{lem:major arcs, small scales}
	We have 
\begin{align*}
\|\sup_{\max\{D, s/\delta\}\le k\le \kappa_s}|f\ast K_{k,s}|\|_{\ell^2(\ZZ^\Gamma)}\lesssim 2^{-s/D}\|f\|_{\ell^2(\ZZ^\Gamma)},
\end{align*}
uniformly in $s\in\NN$ and $f\in\ell^2(\ZZ^\Gamma)$.
\end{lemma}

\begin{lemma}\label{lem:major arcs, large scales}
	We have 
\begin{align*}
\|\sup_{k\ge \kappa_s}|f\ast K_{k,s}|\|_{\ell^2(\ZZ^\Gamma)}\lesssim 2^{-s/D}\|f\|_{\ell^2(\ZZ^\Gamma)},
\end{align*}
uniformly in $s\in\NN$ and $f\in\ell^2(\ZZ^\Gamma)$.
\end{lemma}

We prove Lemmas~\ref{lem:minor arcs}, \ref{lem:major arcs, small scales} and \ref{lem:major arcs, large scales}  in the next three subsections. 
\subsection{Proof of Lemma \ref{lem:minor arcs}}
Let $\mathcal{K}_k^c f=f \ast K_k^c.$
We use a high order $TT^*$ argument. Let
\begin{equation*}
\big( (\mathcal{K}_k^c)^*\mathcal{K}_k^c\big)^r f =: f\ast A_k^{c,r}
\end{equation*}
and notice that it suffices to prove that 
\begin{equation}\label{eq:6.1}
\|A_{k}^{c,r}\|_{\ell^1(\ZZ^\Gamma)}\lesssim 2^{-k}.
\end{equation}
Indeed, assuming \eqref{eq:6.1}, we obtain
\begin{align*}
\|\mathcal{K}_k^c\|_{\ell^2(\ZZ^\Gamma)\rightarrow \ell^2(\ZZ^\Gamma)}=\big\|\big((\mathcal{K}_k^c)^*\mathcal{K}_k^c\big)^r\big\|^{1/2r}_{\ell^2(\ZZ^\Gamma)\rightarrow \ell^2(\ZZ^\Gamma)}\le \|A_{k}^{c,r}\|_{\ell^1(\ZZ^\Gamma)}^{1/2r}\lesssim  2^{-k/2r}\le 2^{-k/D},
\end{align*} 
where the last inequality is a consequence of our definition of $D$; see item \ref{item3}.

\begin{proof}[Proof of inequality \eqref{eq:6.1}]
The general representation of the kernel in Section \ref{subsec:TT*} shows that
\begin{align}
\label{eq:6.3}
A_k^{c,r}(y):=
\eta_{\le 2\delta k}(2^{-k}\circ y)\int_{\TT^\Gamma}
\ex (y\cdot\theta)\Big|\sum_{g\in\ZZ^\Gamma}K_k^c(g)\ex(- g\cdot\theta )\Big|^{2r}d\theta.
\end{align}
Here we could multiply the last integral by the function $\eta_{\le 2\delta k}(2^{-k}\circ (\cdot))$, since the latter is constantly equal to $1$ on the support of $A_k^{c,r}$; see item \ref{item3}.
Observe that 
\begin{align}\label{eq:6.2}
\sum_{g\in\ZZ^\Gamma}K_k^c(g) \ex( - g\cdot\theta )
=
\int_{\TT^\Gamma} S_k(\xi)\Xi_{k}^c(\xi)\Phi_k(\theta-\xi)\,d\xi,
\end{align}
where $\Phi_k$ is defined in \eqref{eq:6.4}.
By applying Weyl's inequality \eqref{eq:56}, we will show that
\begin{equation}\label{eq:Sk}
|S_k(\xi)| 
\lesssim 
2^{-k\delta\tau/(4 |\Gamma|)}, 
\qquad \xi\in \supp  \Xi_{k}^c, \quad k \ge D.
\end{equation}
Indeed, fix $\xi \in \supp  \Xi_{k}^c$. 
By Dirichlet's principle for any $\gamma\in\Gamma$ there exist integers  $a_\gamma$ and $1\le q_\gamma\le 2^{k(|\gamma|-\delta/(2|\Gamma|))}$ such that $\gcd(a_\gamma, q_\gamma)=1$ and 
\begin{equation*}
|\xi_\gamma-{a_\gamma}/{q_\gamma}|
\le
q_\gamma^{-1} 2^{-k(|\gamma|-\delta/(2|\Gamma|))}
\le
q_\gamma^{-2}.
\end{equation*}  
Notice that if for each $\gamma\in\Gamma$ we had 
$q_\gamma \le 2^{k\delta/(2|\Gamma|)}$ then there would exist $1\le q\le \lcm(q_\gamma:\gamma\in\Gamma)\le 2^{\delta k/2}$ and $(a'_\gamma)_{\gamma\in\Gamma}$ such that $\gcd\big(q, \gcd(a'_\gamma : \gamma \in \Gamma)\big) = 1$ and 
\begin{equation*}
|\xi_\gamma-{a'_\gamma}/{q}|
\le
 2^{-k(|\gamma|-\delta/(2|\Gamma|))}, \qquad \gamma\in\Gamma.
\end{equation*}
This would yield that $\xi\notin\supp\Xi_{k}^c$, which is impossible. This contradiction shows that there exists $\gamma_0\in\Gamma$ for which $2^{k \delta/(2|\Gamma|)} \le q_{\gamma_0} \le 2^{k(|\gamma_0|-\delta/(2|\Gamma|))}$ and 
$|\xi_{\gamma_0}-a_{\gamma_0}/q_{\gamma_0} |\le q_{\gamma_0}^{-2}$. 

Therefore, using \eqref{eq:56} and the fact that $|\Omega_k \cap \ZZ^d| \simeq 2^{kd}$, we obtain 
\eqref{eq:Sk}.
Next, combining \eqref{eq:6.3} with \eqref{eq:6.2}, \eqref{eq:Sk} and  \eqref{eq:FkL1} we have
\begin{align*}
|A_k^{c,r}(y)|
\lesssim
	\eta_{\le 2\delta k}(2^{-k}\circ y)
	2^{-r k \tau \delta/(2|\Gamma|)}, \qquad y \in \ZZ^{\Gamma}.
\end{align*}
Since $r \tau \ge \delta^{-4}$ by item \ref{item2}, this proves
\eqref{eq:6.1} and consequently the proof of Lemma~\ref{lem:minor
arcs} is finished.
\end{proof}

\subsection{Proof of Lemma \ref{lem:major arcs, small scales}}

We use the following variant of the Rademacher--Menshov inequality 
\begin{align}
\label{eq:16}
\sup_{A\le j \le 2^W}|a_j-a_{j_0}|\le \sqrt{2}\sum_{i=0}^W\Big(\sum\limits_{\substack{j=0\\ j2^i\ge A}}^{2^{W-i}-1}|a_{(j+1)2^i}-a_{j2^i}|^2\Big)^{1/2},
\end{align}
with an arbitrary positive constant $A\in\RR_+$ and any
$A\le j_0 \le 2^W$.  This inequality immediately follows from
\cite[Lemma~2.5, p.\ 534]{MSZ2}.  Using this estimate together with a
standard linearization argument (Khintchine's inequality followed by a
dyadic decomposition with respect to the variable $k$) reduces Lemma
\ref{lem:major arcs, small scales} to showing the following bound
\begin{align}\label{eq:7.1}
\Big\| \sum_{k\in[J,2J]} \omega_k (f\ast H_{k,s}) \Big\|_{\ell^2(\ZZ^\Gamma)} 
\lesssim 
2^{-2s/D}\|f\|_{\ell^2(\ZZ^\Gamma)},
\end{align}
where $J\ge\max \{D, s/\delta\}$, $|\omega_k|\le 1$ and $H_{k,s} := K_{k+1,s}-K_{k,s}.$
In fact, we also need to prove that
$\|f\ast K_{\kappa_{s},s} \|_{\ell^2(\ZZ^\Gamma)} \lesssim 2^{-s/D}\|f\|_{\ell^2(\ZZ^\Gamma)}$, $f \in \ell^2(\ZZ^\Gamma)$. 
However, this follows from Lemma~\ref{lem:major arcs, large scales}, whose proof is given in Subsection~\ref{subsec:989} below.

Using \eqref{id:604}, we decompose $H_{k,s} = H_{k,s}^1 + H_{k,s}^2 + H_{k,s}^3$, where
\begin{align*}
H_{k,s}^1(x) &:=\phi_k(x)\int_{\TT^\Gamma} \ex(x\cdot\xi)[S_{k+1}(\xi)-S_k(\xi)]\Xi_{k,s}(\xi)\, d\xi, \\
H_{k,s}^2(x) &:=(\phi_{k+1}(x)-\phi_k(x))\int_{\TT^\Gamma} \ex(x\cdot\xi)S_{k+1}(\xi)\Xi_{k,s}(\xi)\, d\xi, \\
H_{k,s}^3(x) 
& :=
\phi_{k+1}(x)\int_{\TT^\Gamma} \ex(x\cdot\xi)S_{k+1} (\xi)[\Xi_{k+1,s}(\xi)-\Xi_{k,s}(\xi)]\, d\xi.
\end{align*}
To prove \eqref{eq:7.1} it suffices to show that
\begin{equation}\label{eq:8.1}
\|f\ast H_{k,s}^l\|_{\ell^2(\ZZ^\Gamma)}\lesssim 2^{-k/D}\|f\|_{\ell^2(\ZZ^\Gamma)}, \qquad l\in\{2,3\}, \quad k\ge \max\{D, s/\delta\}, \quad s\ge 0,
\end{equation}
and
\begin{equation}\label{eq:8.2}
\Big\| \sum_{k\in[J,2J]} \omega_k (f\ast H_{k,s}^1) \Big\|_{\ell^2(\ZZ^\Gamma)} \lesssim 2^{-2s/D}\|f\|_{\ell^2(\ZZ^\Gamma)}, 
\end{equation}
uniformly in $J\ge \max\{D, s/\delta\}, s\ge 0$ and $|\omega_k|\le 1.$
The proof of inequality \eqref{eq:8.1} for $l\in\{2, 3\}$ will be based on exploiting a high order $TT^*$ argument, as we expect to produce exponential decay in $k$. The proof of inequality \eqref{eq:8.2} will be more delicate as the decay in $k$ is not available, but Cotlar--Knapp--Stein lemma will be critical, see Lemma \ref{lem:1}. We provide the details in the next three subsections.

\subsubsection{\textbf{Step 1: Proof of \eqref{eq:8.1} for $l=2$}} Inequality \eqref{eq:8.1} for $l=2$ will follow if we show that
\begin{align}
\label{eq:7}
\|H_{k,s}^2 \|_{\ell^1(\ZZ^\Gamma)}
\lesssim 
2^{-k}.
\end{align}
\begin{proof}[Proof of inequality \eqref{eq:7}]
Notice that if $x\in \supp\big(\phi_{k+1}-\phi_k\big)$, then there exists $\gamma_0\in\Gamma$
such that
\begin{equation} \label{id:301}
|x_{\gamma_0}|\gtrsim 2^{k(|\gamma_0|+\delta)}.
\end{equation}
We will integrate by parts with respect to variable $x_{\gamma_0}$. Note that for any  fixed $N\in\NN$, we have
\begin{align*}
|\partial_{\gamma_0}^N \Xi_{k,s}(\xi)|\lesssim_{N} 2^{k(|\gamma_0|-\delta)N}, \qquad
|\partial_{\gamma_0}^N S_{k+1}(\xi)|\lesssim_{N} 2^{k|\gamma_0|N},
\qquad \xi\in\TT^\Gamma, \quad k\ge \max\{D, s/\delta\}.
\end{align*}
Therefore,
$|\partial_{\gamma_0}^N \big(S_{k+1}(\xi)\Xi_{k,s}(\xi)\big)| 
\lesssim_{N} 2^{k|\gamma_0|N}$,
and integration by parts, for any  fixed $M\in \NN$, gives 
\begin{align*} 
\Big|\int_{\TT^\Gamma} \ex(x\cdot\xi) S_{k+1}(\xi)\Xi_{k,s}(\xi)\, d\xi\Big|
\lesssim_{M} 
|x_{\gamma_0}|^{-M} \int_{\TT^\Gamma} \big|\partial_{\gamma_0}^M \big(S_{k+1}(\xi)\Xi_{k,s}(\xi)\big) \big| \, d\xi 
\lesssim_{M} 2^{-\delta k M}, 
\end{align*}
provided that \eqref{id:301} is satisfied. 
Letting $M=\lfloor (4|\Gamma|d_{\Gamma}+1)/\delta\rfloor +1$, we obtain
\begin{align*}
\|H_{k,s}^2 \|_{\ell^1(\ZZ^\Gamma)}
\lesssim 
\Big(\prod_{\gamma\in\Gamma} 2^{(k+1)(|\gamma|+\delta)} \Big)2^{-\delta k M}
\lesssim
2^{-k}, 
\end{align*}
which implies \eqref{eq:7}, and the proof of \eqref{eq:8.1} for $l=2$ is completed.
\end{proof}

\subsubsection{\textbf{Step 2: Proof of \eqref{eq:8.1} for $l=3$}}
Let $\mathcal{H}_{k,s}^3 f:=H_{k,s}^3\ast f$. Computing a high power
of $TT^*$ as in the proof of Lemma \ref{lem:minor arcs}, we obtain
\begin{equation*}
\big( (\mathcal{H}_{k,s}^3 )^*\mathcal{H}_{k,s}^3\big)^r f=:f\ast A_{k,s}^{3,r},
\end{equation*} 
where
\begin{equation*}
A_{k,s}^{3,r}(y):=\eta_{\le 2\delta k}(2^{-k}\circ y)\int_{\TT^\Gamma} \ex(y\cdot\theta)\Big|\sum_{g\in\ZZ^\Gamma} H_{k,s}^3(g)\ex(-g\cdot\theta )\Big|^{2r}d\theta.
\end{equation*}
Clearly it suffices to prove
\begin{align}
\label{eq:8}
\|A_{k,s}^{3,r}\|_{\ell^1(\ZZ^\Gamma)}\lesssim 2^{-k}.
\end{align}
\begin{proof}[Proof of inequality \eqref{eq:8}]
Recalling $\Phi_k$ given in \eqref{eq:6.4}, we have
\begin{align*}
\sum_{g\in\ZZ^\Gamma}{H}_{k,s}^3(g) \ex(-g\cdot\theta )=\int_{\TT^\Gamma} S_{k+1}(\xi)(\Xi_{k+1,s}(\xi)-\Xi_{k,s}(\xi)) \Phi_{k+1}(\theta-\xi)\,d\xi.
\end{align*}
Proceeding as in the proof of Lemma~\ref{lem:minor arcs} we see that the proof of \eqref{eq:8}  is reduced  to showing that $|S_{k+1} (\xi)| \lesssim 2^{-k\delta\tau/(4 |\Gamma|)}$, uniformly in 
$\xi\in \supp(\Xi_{k,s} - \Xi_{k+1,s})$ and $k\ge \max\{D, s/\delta\}$. This follows from \eqref{eq:Sk}, since
\[
\supp(\Xi_{k,s} - \Xi_{k+1,s}) \subseteq \supp \Xi_{k+1}^c.
\]
This concludes \eqref{eq:8} and the proof of \eqref{eq:8.1} for $l=3$ is completed.
\end{proof}

\subsubsection{\textbf{Step 3: Proof of \eqref{eq:8.2}}}
To this end again we will compute a high order $TT^*$. Let
\begin{equation*}
\mathcal{H}_{k,s}^1 f:=f\ast H_{k,s}^1.
\end{equation*}

For $k \in \NN$ and $\xi \in \RR^{\Gamma}$ let us introduce the following auxiliary functions
\begin{align} \label{eq:9.1}
	\begin{split}
	J(\xi) &:=
	\frac{1}{|\Omega|}\int_{\Omega} \ex(-(u)^\Gamma\cdot \xi)\, du, \\
	J_k(\xi)
	& := J(2^{k} \circ \xi)
	=\frac{1}{|\Omega|}\int_{\Omega} \ex(-(2^k u)^\Gamma\cdot \xi)\, du
	= \frac{1}{|\Omega_k|}\int_{\Omega_k} \ex(-(u)^\Gamma\cdot \xi)\, du,  \\
	\widetilde{J}(\xi)
	& :=
	J(2 \circ \xi) - J(\xi), \\
	\widetilde{J_k}(\xi)
	& :=  
	\widetilde{J}(2^{k} \circ \xi)
	=
	J_{k+1}(\xi)-J_k(\xi) 
	=
	\int_{\RR^d}\rho(u) \ex\big(-(u)^\Gamma\cdot (2^k\circ \xi)\big)\, du, \\
	\rho(u)
	& :=
	\frac{1}{|2\Omega|}\mathbbm{1}_{2\Omega} (u) 
	-
	\frac{1}{|\Omega|}\mathbbm{1}_\Omega(u) , \qquad u \in \RR^d,\\
	\mathcal{J} (\xi)
		& :=
		|\widetilde{J}(\xi)|^{2r}
		=
		|\widetilde{J_k}(2^{-k}\circ \xi)|^{2r}.
\end{split}
\end{align}
For future reference notice that $\rho\in L^1(\RR^d)$ with $\int_{\RR^d} \rho(u)\,du =0$ and $\supp \rho\subseteq \{u\in \RR^d: |u|\le 2\}$.

We begin with approximating the multiplier $S_k$ on major arcs.
\begin{lemma}\label{fact:2.2}
Let $|2^k\circ\xi|\le 2^{k/4}$, $q\le 2^{k/4}$ and $a \in \ZZ^{\Gamma}$ be such that $\gcd(a,q) = 1$. Then $E_{k}(\xi):=S_k(a/q+\xi)-G(a/q)J_k(\xi)$ satisfies
\begin{equation*}
|E_{k}(\xi)|\lesssim 2^{-k/2},
\end{equation*}
where $G(a/q)$ and $J_k(\xi)$ are defined in \eqref{def:Gauss} and \eqref{eq:9.1}, respectively.
\end{lemma}

\begin{proof}
Using \cite[Proposition~4.16]{advances} we see that
$|\Omega_k \cap \ZZ^d| = |\Omega_k| + O(2^{k(d-1)})$ and consequently $|\Omega_k \cap \ZZ^d| \simeq |\Omega_k|$.  Therefore an
application of \cite[Proposition~4.18]{advances}, (with
$N=2^k, \Omega=\Omega_k$, $\mathcal{K} \equiv 1$ and $\epsilon_\gamma=1$ for each
$\gamma\in\Gamma$), gives
\begin{align*}
|E_{k}(\xi)|
&=
|\Omega_k|^{-1}\Big|\sum_{n\in\Omega_k \cap \ZZ^d}
\ex(-(n)^\Gamma\cdot \xi)-G(a/q)\int_{\Omega_k}\ex(-(x)^\Gamma\cdot \xi)\,dx\Big|
+ O(2^{-k})
\\
&\lesssim 
q/2^k +q\sum_{\gamma\in\Gamma}|\xi_\gamma| 2^{k(|\gamma|-1)} + 2^{-k},
\end{align*}
and the last expression is dominated by $2^{-k/2}$ and the proof follows.
\end{proof}

\begin{lemma}\label{lem:2.1}
Let $s\ge 0$ and $k\ge \max\{D, s/\delta\}$. Then 
\begin{align*}
\big((\mathcal{H}_{k,s}^1)^*\mathcal{H}_{k,s}^1\big)^r f= f\ast (B_{k,s}^r+E_{k,s}^r), 
\end{align*}
where $\|E_{k,s}^r\|_{\ell^1(\ZZ^\Gamma)}\lesssim 2^{-k/8}$ is an error term and
\begin{align*}
B_{k,s}^r(x)
&:=
\Big(\sum_{a/q\in R_s^\Gamma \cap [0,1)^{\Gamma}} \ex(x\cdot a/q) |G(a/q)|^{2r} \Big) \\
& \qquad \qquad \qquad \times
\eta_{\le 2\delta k}(2^{-k}\circ x)
\Big(\prod_{\gamma\in\Gamma}2^{-k|\gamma|}\Big)\int_{\RR^\Gamma}\eta_{\le \delta k}(\theta) \mathcal{J} (\theta)\ex((2^{-k}\circ x)\cdot \theta)\, d\theta,
\end{align*}
with  $G(a/q)$ and $\mathcal{J}$ defined in \eqref{def:Gauss} and \eqref{eq:9.1}, respectively.
\end{lemma}

Before we prove Lemma~\ref{lem:2.1}, observe that
 for any fixed $\alpha \in \NN^{\Gamma}$ we have the estimate
\begin{equation}\label{eq:13.1}
	\big|D^\alpha \mathcal{J}(\theta)\big|\lesssim_\alpha (1+|\theta|)^{-1/\delta^2}, 
	\qquad \theta \in \RR^{\Gamma}.
\end{equation}
Inequality \eqref{eq:13.1} follows from  Lemma~\ref{lem:A2} combined with \eqref{eq:tau} and the relation from item \ref{item2}.

\begin{proof}[Proof of Lemma~\ref{lem:2.1}]
Let $\big((\mathcal{H}_{k,s}^1)^*\mathcal{H}_{k,s}^1\big)^r f= f\ast H_{k,s}^r$.
By the formula derived in Section \ref{subsec:TT*}, we obtain
\begin{align} \label{idd::A2}
H_{k,s}^r(x):=\eta_{\le 2\delta k}(2^{-k}\circ x)\int_{\TT^\Gamma} \ex(x\cdot \theta) \Big|\sum_{g\in\ZZ^\Gamma} H_{k,s}^1(g) \ex(-g\cdot\theta) \Big|^{2r}\, d\theta.
\end{align}
Recalling $\Phi_k$ defined in \eqref{eq:6.4}, and setting
$\widetilde{S_k} :=S_{k+1}-S_k$, we observe that
\begin{equation*}
\sum_{g\in\ZZ^\Gamma} H_{k,s}^1(g) \ex(-g\cdot\theta)
=
\int_{\TT^\Gamma} \Phi_k(\theta-\xi) \widetilde{S_k}(\xi)\Xi_{k,s}(\xi)\,d\xi.
\end{equation*}

If $\xi\in\supp\Xi_{k,s}$, then there exists a fraction
$a/q\in R_s^\Gamma$, which lies close to $\xi$.  In view of
\eqref{eq:6.5} we see that $\Phi_k(\theta-\xi)$ decays rapidly if
$|\theta-\xi|$ is separated from zero. Combining these two
observations, the integration in $d\theta$ in \eqref{idd::A2} can be
localized to small neighborhoods of fractions from $R_s^\Gamma$.

More precisely, we can write
\begin{equation*}
H_{k,s}^r(x)=H_{k,s}^{r,1}(x)+E_{k,s}^{r,1}(x),
\end{equation*}
with $ \|E_{k,s}^{r,1}\|_{\ell^1(\ZZ^\Gamma)}\lesssim 2^{-k}$
and 
\begin{align}
H_{k,s}^{r,1}(x)\nonumber
&:=
\eta_{\le 2\delta k}(2^{-k}\circ x)
\int_{\TT^\Gamma} \ex(x\cdot \theta) \sum_{a/q\in R_s^\Gamma} \eta_{\le 2\delta k} \big(2^k\circ (\theta-a/q)\big) \Big|\sum_{g\in\ZZ^\Gamma} H_{k,s}^1(g) \ex(-g\cdot\theta)\Big|^{2r}\, d\theta
\\
&=
\eta_{\le 2\delta k}(2^{-k}\circ x) \sum_{a/q\in R_s^\Gamma \cap [0,1)^{\Gamma} } \ex(x\cdot a/q)
\int_{\RR^\Gamma} \ex(x\cdot \theta)\eta_{\le 2\delta k} (2^k\circ \theta) |I_{a/q}(\theta)|^{2r}\, d\theta,\label{eq:*}
\end{align}
where for $a/q \in R_s^\Gamma \cap [0,1)^{\Gamma}$ and $\theta \in \RR^\Gamma$ we set
$$
I_{a/q}(\theta)
:=
\sum_{g\in\ZZ^\Gamma} H_{k,s}^1(g) \ex\big(-g\cdot(\theta+a/q)\big).
$$
Next, for fixed $a/q$ and $|2^k\circ \theta|\lesssim 2^{2\delta k}$ we will approximate $I_{a/q}(\theta)$. 
Using Lemma~\ref{fact:2.2},  we obtain
\begin{align*} 
I_{a/q}(\theta)
&=
\sum_{g\in\ZZ^\Gamma} \phi_k(g) \ex(-g\cdot (\theta+a/q)) \int_{\TT^\Gamma} \ex(g\cdot\xi)\widetilde{S_k}(\xi)\Xi_{k,s}(\xi)\, d\xi
\\
&=\sum_{g\in\ZZ^\Gamma}  \phi_k(g) \ex\big(-g\cdot (\theta+a/q)\big)
\sum_{b/q' \in R_s^\Gamma \cap [0,1)^{\Gamma}} \ex(g\cdot b/q') 
\int_{\RR^\Gamma} \ex(g\cdot \xi) \widetilde{S_k}(\xi+b/q') \eta_{\le \delta k} (2^k\circ \xi) \, d\xi \\
&=
\sum_{g\in\ZZ^\Gamma}  \phi_k(g) 
\sum_{b/q' \in R_s^\Gamma \cap [0,1)^{\Gamma}} \ex\big(-g\cdot (\theta+a/q - b/q')\big) G(b/q')
\int_{\RR^\Gamma} \ex(g\cdot \xi) \widetilde{J_k}(\xi) \eta_{\le \delta k} (2^k\circ \xi) \, d\xi \\
& \quad
+ O(2^{-k/4}). 
\end{align*}
Changing the variable $\xi\mapsto 2^{-k}\circ \xi$, we have
\begin{align*}
\int_{\RR^\Gamma} \ex(g\cdot\xi)\widetilde{J_k}(\xi)\eta_{\le \delta k}(2^k\circ \xi)\, d\xi
=
\Big(\prod_{\gamma\in\Gamma} 2^{-k|\gamma|}\Big) \int_{\R^d}\rho(x) \widehat{\eta_{\le \delta k}}((x)^\Gamma-2^{-k}\circ g)\,dx,
\end{align*}
where $\widehat{\eta_{\le \delta k}}:=\mathcal F_{\RR^\Gamma}\eta_{\le \delta k}$.
Since for any fixed $M \in \NN$ we have 
\begin{equation*}
	\big| \widehat{\eta_{\le \delta k}} \big((x)^\Gamma-2^{-k}\circ g \big) \big|
	\lesssim_{M} 
	2^{-Mk} \Big(1 + 2^{\delta k} \big| (x)^\Gamma-2^{-k}\circ g \big| \Big)^{-M},
	\qquad |x| \le 2, \quad |2^{-k}\circ g| \ge 2^{\delta k},
\end{equation*}
we can replace $\phi_k(g)$ with $1$ in the formula approximating $I_{a/q}(\theta)$. Therefore, we conclude
\begin{align} \label{id:601}
	\begin{split}
	I_{a/q}(\theta)
& = 
O(2^{-k/4}) 
+
	\sum_{b/q' \in R_s^\Gamma \cap [0,1)^{\Gamma}} G(b/q')
	\Big(\prod_{\gamma\in\Gamma} 2^{-k|\gamma|}\Big) \\
& \qquad \qquad \qquad \qquad \times	
	\sum_{g\in\ZZ^\Gamma} 
	\ex\big(-g\cdot (\theta+a/q - b/q')\big) 
	\int_{\R^d}\rho(x) \widehat{\eta_{\le \delta k}}((x)^\Gamma-2^{-k}\circ g)\,dx. 
\end{split}
\end{align}
Applying the Poisson summation formula to the sum over $g\in\ZZ^\Gamma$ above, we may write
\begin{align} \label{id:602}
		\begin{split}
& \Big(\prod_{\gamma\in\Gamma} 2^{-k|\gamma|}\Big)	\sum_{g\in\ZZ^\Gamma} 
\ex\big(-g\cdot (\theta+a/q - b/q')\big)
\widehat{\eta_{\le \delta k}} ((x)^\Gamma-2^{-k}\circ g) \\
& =
\sum_{u \in\ZZ^\Gamma} \ex \big[-(x)^\Gamma\cdot \big(2^k\circ (\theta + a/q - b/q' + u) \big) \big]
\eta_{\le \delta k} \big(2^k\circ(\theta+a/q-b/q' + u) \big) \\
& =
\ind{a/q = b/q'} \ex \big(-(x)^\Gamma\cdot (2^k\circ \theta ) \big)
\eta_{\le \delta k} (2^k\circ \theta ).
\end{split}
\end{align}
In the last identity we used the fact that $|2^k\circ \theta|\lesssim 2^{2\delta k}$ and $|2^k\circ(\theta+a/q-b/q' + u)| \lesssim 2^{\delta k}$ 
force $|2^k\circ(a/q-b/q' + u)| \lesssim 2^{2\delta k}$. Since $a/q,b/q' \in R_s^\Gamma \cap [0,1)^{\Gamma}$, the latter is true only if $u=0$ and $a/q = b/q'$.

Combining \eqref{id:601} with \eqref{id:602}, we obtain
\begin{align*}
	I_{a/q}(\theta)
=
O(2^{-k/4})
+
G(a/q)\eta_{\le \delta k}(2^{k}\circ \theta)\widetilde{J_k}(\theta),
\end{align*}
and consequently 
\begin{align*}
|I_{a/q}(\theta)|^{2r}
	=
	O(2^{-k/4})
	+
	|G(a/q) \eta_{\le \delta k}(2^{k}\circ \theta)\widetilde{J_k}(\theta)|^{2r}.
\end{align*}
Using this with \eqref{eq:*}, the identity
$\eta_{\le 2\delta k}\eta_{\le\delta k}=\eta_{\le \delta k}$ and  \eqref{eq:9.1}
we see that up to an acceptible error $H_{k,s}^{r,1}$ is equal to
\begin{align*}
	H_{k,s}^{r,2}(x)
	=
\eta_{\le 2\delta k}(2^{-k}\circ x) \sum_{a/q\in R_s^\Gamma \cap [0,1)^{\Gamma} } \ex(x\cdot a/q) |G(a/q)|^{2r}
\int_{\RR^\Gamma} \ex(x\cdot \theta) | \eta_{\le \delta k}(2^{k}\circ \theta)\widetilde{J_k}(\theta)|^{2r}\, d\theta.
\end{align*}
Finally, using \eqref{eq:13.1} we may replace 
$\big(\eta_{\le \delta k}(2^k\circ \theta)\big)^{2r}$ with $\eta_{\le \delta k}(2^k\circ \theta)$. Then the change of variables $2^k\circ \theta\mapsto \theta$ concludes the proof of Lemma~\ref{lem:2.1}.
\end{proof}

We return to the proof of \eqref{eq:8.2}. By the Cotlar--Knapp--Stein lemma it suffices to show the following lemma.

\begin{lemma}\label{lem:2.3}
Let $k,j\ge \max\{D, s/\delta\}$ and $j\in[k/2,k]$. Then
\begin{equation*}
\|\mathcal{H}_{j,s}^1(\mathcal{H}_{k,s}^1)^*\|_{\ell^2(\Z^\Gamma)\rightarrow \ell^2(\Z^\Gamma)}\lesssim 2^{-4s/D}2^{-|k-j|/D}.
\end{equation*}
\end{lemma}

\begin{proof}
We will consider two cases.

\paragraph{\textbf{Special case $j=k$}}
Using Lemma \ref{lem:2.1}, it suffices to show $\|B_{k,s}^r\|_{\ell^1(\ZZ^\Gamma)}\lesssim 2^{-s}$.
We have 
\begin{align*}
B_{k,s}^r(x)
=
\sum_{a/q\in R_s^\Gamma \cap [0,1)^{\Gamma}} |G(a/q)|^{2r} X_{k,a/q}^r(x), 
\end{align*}
where
\begin{align*}
X_{k,a/q}^r(x) & := \ex(x\cdot a/q) X_k^r(x), \\
X_k^r(x)
& :=
\eta_{\le 2\delta k}(2^{-k}\circ x)\Big(\prod_{\gamma\in\Gamma} 2^{-k|\gamma|}\Big)
\int_{\RR^\Gamma} \eta_{\le \delta k}(\theta)
\mathcal{J} (\theta) \ex\big((2^{-k}\circ x)\cdot \theta\big)\,d\theta.
\end{align*}
Using \eqref{eq:13.1} and integrating by parts, for a fixed $M \in \NN$, we have
\begin{align}\label{eq:16.3}
|X_k^r(x)|+\sum_{\gamma\in\Gamma}2^{k|\gamma|} |\partial_{x_\gamma} X_k^r(x)|
\lesssim_{M} 
\Big(\prod_{\gamma\in\Gamma} 2^{-k|\gamma|} \Big)(1+|2^{-k}\circ x|)^{-M}.
\end{align}
The above estimate,  Lemma \ref{lem:A4}, and items \ref{item1} and \ref{item2} lead to
\begin{equation*}
\|B_{k,s}^r\|_{\ell^1(\ZZ^\Gamma)}
\le 
\|X_k^r\|_{\ell^1(\ZZ^\Gamma)}\sum_{a/q\in R_s^\Gamma \cap [0,1)^{\Gamma}}|G(a/q)|^{2r}\lesssim 2^{-2r\tau s}
|R_s^\Gamma \cap [0,1)^{\Gamma} |\le 2^{-s}.
\end{equation*}

\paragraph{\textbf{General case $j\in[k/2,k]$} }
From the previous case, in particular, we know that 
$$
\|\mathcal{H}_{j,s}^1\|_{\ell^2(\Z^\Gamma)\rightarrow \ell^2(\Z^\Gamma)}\lesssim 1,
$$ 
so we can reduce our task to showing that
\begin{equation*}
\|\mathcal{H}_{j,s}^1[(\mathcal{H}_{k, s}^1)^*\mathcal{H}_{k,s}^1]^r\|_{\ell^2(\Z^\Gamma)\rightarrow \ell^2(\Z^\Gamma)}\lesssim 2^{-s}2^{-\delta(k-j)/2},
\end{equation*}
uniformly in $k,j\ge \max\{D, s/\delta\}$ and $j\in[k/2,k]$.

We first use Lemma \ref{lem:2.1}. The contribution from the error term can be easily treated as follows
\begin{align*}
\|  E_{k,s}^r \ast H_{j,s}^1 \|_{\ell^1(\ZZ^\Gamma)} 
&\le 
\|E_{k,s}^r\|_{\ell^1(\ZZ^\Gamma)} \|H_{j,s}^1\|_{\ell^1(\ZZ^\Gamma)}
\lesssim 
2^{-k/8} \|\phi_j\|_{\ell^1(\ZZ^\Gamma)}|\supp \Xi_{j,s}|\lesssim 2^{-k/8} 2^{4k|\Gamma|\delta}\le 2^{-k/16}.
\end{align*}
Therefore, it suffices to prove that
\begin{equation}\label{eq:14.1}
\|B_{k,s}^r\ast H_{j,s}^1\|_{\ell^1(\ZZ^\Gamma)}\lesssim  2^{-s}2^{-\delta(k-j)/2}.
\end{equation}
Using Lemma \ref{fact:2.2}, we see that
\begin{equation*}
H_{j,s}^1(x)=\sum_{b/q' \in R_s^\Gamma \cap [0,1)^{\Gamma}} G(b/q') Y_{j,b/q'}(x)+E_{j,s}(x),
\end{equation*}
where $\|E_{j,s}\|_{\ell^1(\ZZ^\Gamma)}\lesssim 2^{-j/4}$ and 
\begin{align*}
Y_{j,b/q'}(x)
& :=
\ex(x\cdot b/q') Y_j(x), \\ 
Y_j(x)
& :=
\phi_j(x)\int_{\RR^\Gamma}\ex(x\cdot \xi) \widetilde{J_j}(\xi)\eta_{\le\delta j}(2^j\circ \xi)\,d\xi.
\end{align*}
Using Lemma~\ref{lem:A4}, we see that 
\begin{align*}
\|B_{k,s}^r\ast H_{j,s}^1\|_{\ell^1(\ZZ^\Gamma)}
&\le 
\sum_{a/q \in R_s^\Gamma \cap [0,1)^{\Gamma}}
\sum_{b/q' \in R_s^\Gamma \cap [0,1)^{\Gamma}}
|G(a/q)|^{2r}\|X_{k,a/q}^r \ast Y_{j,b/q'}\|_{\ell^1(\ZZ^\Gamma)}+O(2^{-s} 2^{-j/4})
\\
&\lesssim 2^{-s}\sup_{a/q, b/q' \in R_s^\Gamma}\|X_{k,a/q}^r \ast Y_{j,b/q'}\|_{\ell^1(\ZZ^\Gamma)}
+
2^{-s}2^{-j/4}.
\end{align*}
Consequently, in order to prove \eqref{eq:14.1} it suffices to check that
\begin{equation}\label{eq:15.1}
\sup_{a/q, b/q' \in R_s^\Gamma}
\|X_{k,a/q}^r \ast Y_{j,b/q'}\|_{\ell^1(\ZZ^\Gamma)}
\lesssim 2^{-\delta(k-j)/2}.
\end{equation}
Fixing $a/q, b/q'\in R_s^\Gamma$,  and splitting the above convolution into classes modulo $Q:=qq' \lesssim 2^{2s}$, we obtain
\begin{align*}
\|X_{k,a/q}^r \ast Y_{j,b/q'}\|_{\ell^1(\ZZ^\Gamma)}
\le
\sum_{x\in\ZZ^\Gamma}\sum_{c\in\ZZ_Q^\Gamma}\Big|\sum_{y\in\ZZ^\Gamma} X_{k}^r(x-Qy-c) Y_j (Qy+c)\Big|\le I_1+I_2,
\end{align*}
where
\begin{align*}
I_1:=&\sum_{x\in\ZZ^\Gamma}\sum_{c\in\ZZ_Q^\Gamma}\sum_{y\in\ZZ^\Gamma} |X_{k}^r(x-Qy-c)-X_{k}^r(x)| |Y_{j}(Qy+c)|,\\
I_2:=&\sum_{x\in\ZZ^\Gamma}\sum_{c\in\ZZ_Q^\Gamma}|X_{k}^r(x)| \Big|\sum_{y\in\ZZ^\Gamma}Y_{j}(Qy+c)\Big|.
\end{align*}
 Note that thanks to the splitting into classes modulo $Q$, we managed to annihilate the exponential factors $\ex((\cdot)\cdot a/q)$ and $\ex((\cdot)\cdot b/q')$ without inserting the absolute value under the sum over $y$. That allowed us to preserve the essential cancelations in the sum $\sum_{y\in\ZZ^\Gamma}Y_{j}(Qy+c)$.

Now we will show that
\begin{align}\label{eq:16.1}
|X_k^r(x-y)-X_k^r(x)|
\lesssim 
2^{-(k-j)}\big(1+|2^{-j}\circ y|\big)^{1+D} 
\Big(\prod_{\gamma\in\Gamma} 2^{-k|\gamma|} \Big) 
\big(1+|2^{-k}\circ x|\big)^{-D},
\end{align}
uniformly in $j\le k$ and $x,y\in \RR^\Gamma$. Using \eqref{eq:16.3} and $j\le k$, we obtain
\begin{align*}
|X_k^r(x-y)-X_k^r(x)|
&\le 
\int_0^1 \sum_{\gamma\in\Gamma}|y_\gamma| |\partial_{x_\gamma} X_k^r(x+ty)|\, dt
\\
&\le 
2^{-(k-j)}\int_0^1 \sum_{\gamma\in\Gamma}2^{-j|\gamma|}|y_\gamma| 2^{k|\gamma|}|\partial_{x_\gamma} X_k^r(x+ty)|\, dt
\\
&\lesssim
2^{-(k-j)}|2^{-j}\circ y| \Big(\prod_{\gamma\in\Gamma} 2^{-k|\gamma|}\Big) \int_0^1 \big(1+|2^{-k}\circ (x+ty)|\big)^{-D}\,dt.
\end{align*}
This implies \eqref{eq:16.1}, since we have an estimate 
\begin{align*}
	\big(1+|2^{-k}\circ (x+ty)|\big)^{-1}
	\lesssim
	\big(1+|2^{-k}\circ x|\big)^{-1}
	\big(1+|2^{-k}\circ y|\big), 
	\qquad x, y \in \RR^{\Gamma}, \quad t \in [0,1].
\end{align*}

Next we show that
\begin{equation}\label{eq:16.2}
|Y_j(y)|\lesssim \Big( \prod_{\gamma\in\Gamma} 2^{-j(|\gamma|-\delta)} \Big) \int_{\R^d}|\rho(u)|\Big(1+\big|2^{\delta j}((u)^\Gamma-2^{-j}\circ y)\big|\Big)^{-4D}\, du.
\end{equation}
Taking $\widehat{\eta_{\le \delta k}}:=\mathcal F_{\RR^\Gamma}\eta_{\le \delta k}$ a direct computation involving the definition of $Y_j$ and \eqref{eq:9.1} gives
\begin{align*}
Y_j(y)
&=
\phi_j(y) \int_{\RR^d}\rho(u)\int_{\RR^\Gamma} \eta_{\le \delta j}(2^j\circ \xi) \ex\big(\xi\cdot y- (u)^\Gamma\cdot (2^j\circ\xi)\big)\,d\xi\,du
\\
&=
\Big( \prod_{\gamma\in\Gamma}2^{-j|\gamma|} \Big) \phi_j(y)\int_{\RR^d}\rho(u)\widehat{\eta_{\le \delta j}}((u)^\Gamma - 2^{-j}\circ y)\,du
\end{align*}
and consequently \eqref{eq:16.2} follows.

Next, combining \eqref{eq:16.1} with \eqref{eq:16.2} and using the fact that
\begin{align*}
1+ |2^{-j}\circ y |
\lesssim
1+ 2^{\delta j} |(u)^\Gamma-2^{-j}\circ y|,
\end{align*}
uniformly in $y \in \RR^{\Gamma}$, $|u| \le 2$ and $j \in \ZZ_+$, we obtain 
\begin{align*}
I_1
&=
\sum_{x\in\ZZ^\Gamma}\sum_{y\in\ZZ^\Gamma} |X_{k}^r(x-y)-X_{k}^r(x)| |Y_{j}(y)| \\
&\lesssim 
2^{-(k-j)} \Big( \prod_{\gamma\in\Gamma} 2^{-k|\gamma|-j(|\gamma|-\delta)} \Big) 
\sum_{x\in\ZZ^\Gamma}\sum_{y\in\ZZ^\Gamma}
\big(1+|2^{-j}\circ y| \big)^{1+D} \big(1+|2^{-k}\circ x| \big)^{-D} \\
&\quad\times 
\int_{\R^d}|\rho(u)|\big(1+ 2^{\delta j} |(u)^\Gamma-2^{-j}\circ y|\big)^{-4D}\, du \\
&\lesssim 
2^{-(k-j)} 
\Big(\prod_{\gamma\in\Gamma} 2^{-j(|\gamma|-\delta)} \Big) 
 \int_{\R^d}|\rho(u)| \sum_{y\in\ZZ^\Gamma} 
 \big(1+ 2^{\delta j} |(u)^\Gamma-2^{-j}\circ y|\big)^{-2D}\, du.
\end{align*}
Since we have
\begin{align*}
	\sum_{y\in\ZZ^\Gamma} 
\big(1+ 2^{\delta j} |A-2^{-j}\circ y|\big)^{-2D}
\lesssim
\Big(\prod_{\gamma\in\Gamma} 2^{j(|\gamma|-\delta)} \Big),
\qquad A \in \RR^{\Gamma},  \quad j \ge 0,
\end{align*}
we obtain $I_1 \lesssim 2^{-(k-j)}$ as required.

Next we focus on $I_2$. Using definition of $Y_j$, we see that 
\begin{align*}
I_2\lesssim \sum_{c\in\ZZ_Q^\Gamma}\big|\sum_{y\in \ZZ^\Gamma} Y_j(Qy+c) \big|
\le 
\sum_{c\in \ZZ_Q^\Gamma}\int_{\RR^\Gamma}|\widetilde{J_j}(\xi)| 
\Big| \sum_{y\in\ZZ^\Gamma}\phi_j(Qy+c) \ex \big((Qy+c)\cdot \xi \big) \Big|\eta_{\le \delta j}(2^j\circ \xi)\,d\xi.
\end{align*}
Applying  Lemma \ref{lem:A1}, we obtain
\begin{equation*}
\Big|\sum_{y\in\ZZ^\Gamma}\phi_j(Qy+c) \ex\big((Qy+c)\cdot \xi \big) \Big|
\lesssim 
Q^{-|\Gamma|} \Big(\prod_{\gamma\in\Gamma} 2^{j(|\gamma|+\delta)}\Big)
\big(1+ 2^{\delta j}|2^j\circ \xi |\big)^{-D},
\end{equation*}
uniformly in $|2^j \circ \xi|\lesssim 2^{j/4}, Q\lesssim 2^{j/8}$ and $c\in\ZZ^\Gamma$.
Moreover, since $\int_{\RR^d} \rho(u)\, du = 0$, we have 
$\widetilde{J_j}(0)=0$ and $|\widetilde{J_j}(\xi)|\lesssim |2^j \circ \xi|$. Consequently, we bound 
\begin{align*}
I_2
\lesssim 
\Big( \prod_{\gamma\in\Gamma} 2^{j(|\gamma|+\delta)} \Big)
\int_{\RR^\Gamma} |2^j \circ \xi| 
\big(1+ 2^{\delta j}|2^j\circ \xi |\big)^{-D} \,d\xi 
\lesssim 
2^{-\delta j}\le 2^{-\delta (k-j)/2}.
\end{align*}
This proves \eqref{eq:15.1} and finishes the proof of Lemma \ref{lem:2.3}.
\end{proof}

As observed before, Lemma \ref{lem:2.3}, together with the Cotlar--Knapp--Stein lemma, yields \eqref{eq:8.2}, which, in turn, combined with \eqref{eq:8.1} for $l\in\{2,3\}$, implies inequality \eqref{eq:7.1}. This gives Lemma \ref{lem:major arcs, small scales} as desired.

\subsection{Proof of Lemma \ref{lem:major arcs, large scales}} \label{subsec:989}
For $Q \in \ZZ_{+}$ we define
$ Q \ZZ^\Gamma :=\{Q x: x\in\ZZ^\Gamma\}$, which is the subgroup of
the multiples of $Q$ in $\ZZ^\Gamma$. Then the group
$\ZZ^\Gamma/Q \ZZ^\Gamma$ of residue classes of modulo $Q$ can be
identified with the set $\ZZ_Q^\Gamma=\{0, 1,\ldots, Q-1\}^\Gamma$.
If $\mathbb G\in\{Q \ZZ^\Gamma, \ZZ/Q \ZZ^\Gamma\}$ and
$f,g\colon \mathbb G\rightarrow \CC$, then the convolution on
$\mathbb G$ is defined in a natural way by setting
\begin{equation*}
f\ast_{\mathbb G}g(x) :=\sum_{y\in \mathbb G}f(x-y)g(y), \qquad x\in \mathbb G.
\end{equation*}

Let $Q_s:=(2^{s+1})!$ and by \ref{item3} notice that
$Q_s\le 2^{\delta k}$ for $k\ge \kappa_s$. Our goal now is to represent the kernel $K_{k,s}$ defined in \eqref{id:604} as a tensor product of two kernels on $Q_s \ZZ^\Gamma$ and $\ZZ^\Gamma_{Q_{s}}=\ZZ^\Gamma/Q_s \ZZ^\Gamma$ plus an acceptable error term.
To this end we introduce new kernels. For $k\ge D$ and $Q\ge 1$ let
\begin{align} \label{id:607}
W_{k,Q}(x)
:=
Q^{|\Gamma|}\phi_k(x)\int_{\RR^\Gamma} \eta_{\le \delta k} (2^k\circ \xi) \ex(x\cdot \xi) J_k(\xi)\,d\xi,\qquad x\in Q\ZZ^\Gamma,
\end{align}
where $J_k$ is defined in \eqref{eq:9.1}. Moreover, we let 
\begin{equation*}
V_s(b)
:=
Q_s^{-|\Gamma|}\sum_{a/q \in R_s^\Gamma \cap [0,1)^{\Gamma} } G(a/q) \ex(b\cdot a/q),\qquad b\in\ZZ^\Gamma.
\end{equation*} 
Observe that $V_s$ is $Q_s$-periodic since $q$ divides $Q_s$ for any $a/q\in R_s^\Gamma$ and $s\le \delta k$.
Moreover,  we set $\mathcal W_{k,Q_s}(f):=f\ast_{Q_s\ZZ^\Gamma} W_{k,Q_s}$  and  $\mathcal V_s(f):=f\ast_{\ZZ^\Gamma_{Q_{s}}} V_s$ as well as
\[
\mathcal W_{\ast,Q}(f):=\sup\limits_{\substack{2^{\delta k}\ge Q \\ k\ge D}}|f\ast_{Q\ZZ^\Gamma} W_{k,Q}|, \qquad Q\ge 1.
\]

To prove Lemma \ref{lem:major arcs, large scales} it suffices to show the following three lemmas.

\begin{lemma}\label{lem:3.1}
	We have
\begin{align*}
\|\sup_{k\ge \kappa_s} |f\ast K_{k,s}|\|_{\ell^2(\ZZ^\Gamma)}
\le
\|\mathcal W_{\ast,Q_s}\|_{\ell^2(Q_s\ZZ^\Gamma)\rightarrow \ell^2(Q_s\ZZ^\Gamma)} 
\|\mathcal V_s\|_{\ell^2(\ZZ^\Gamma_{Q_{s}})\rightarrow \ell^2(\ZZ^\Gamma_{Q_{s}})}\|f\|_{\ell^2(\ZZ^\Gamma)}
+O(2^{-s/8}\|f\|_{\ell^2(\ZZ^\Gamma)}),
\end{align*}
uniformly in $s\ge 0$ and $f\in \ell^2(\ZZ^\Gamma)$.
\end{lemma} 

\begin{lemma}\label{lem:3.2}
	We have
\begin{equation*}
\|\mathcal V_s(f)\|_{\ell^2(\ZZ^\Gamma_{Q_{s}})}
\le 
2^{-s/D}\|f\|_{\ell^2(\ZZ^\Gamma_{Q_{s}})},
\end{equation*}
uniformly in $s\ge 0$ and $f\in \ell^2(\ZZ^\Gamma_{Q_{s}})$.
\end{lemma} 

\begin{lemma}\label{lem:3.3}
	We have
\begin{equation*}
\|\mathcal W_{\ast,Q}(f)\|_{\ell^2(Q\ZZ^\Gamma)}\lesssim \|f\|_{\ell^2(Q\ZZ^\Gamma)},
\end{equation*}
uniformly in $Q\ge 1$ and $f\in \ell^2(Q\ZZ^\Gamma)$.
\end{lemma} 
Notice that Lemmas \ref{lem:3.1}, \ref{lem:3.2} and \ref{lem:3.3} imply Lemma \ref{lem:major arcs, large scales}. Now we focus on proving the above lemmas.

\begin{proof}[Proof of Lemma \ref{lem:3.1}]
First we want to show that 
\begin{align}\label{eq:20.1}
K_{k,s}(Q_s y+b)=W_{k, Q_s}(Q_s y)V_s(b)+E_{k,s}(y,b),
\end{align}
uniformly in $k\ge \kappa_s$, $s \ge 0$,  $y\in\ZZ^\Gamma$, $|b|\lesssim Q_s\le 2^{\delta k}$, with $E_{k,s}$ satisfying 
\begin{align} \label{id:605}
|E_{k,s}(y,b)|
\lesssim 
2^{-k/4} \Big( \prod_{\gamma\in\Gamma} 2^{-k|\gamma|}\Big)\eta_{\le2\delta k}(2^{-k}\circ y).
\end{align}
To this end recall, see \eqref{id:604},  that
\begin{align*}
K_{k,s}(x)
=
\phi_k(x)\sum_{a/q\in R_s^\Gamma \cap [0,1)^{\Gamma} } \ex(x\cdot a/q) 
\int_{\RR^\Gamma} \ex(x\cdot\xi) S_k(\xi+a/q)\eta_{\le \delta k}(2^k \circ \xi)\,d\xi.
\end{align*}
Since $Q_s\le 2^{\delta k}$ and $s \le \delta k$ for $k\ge \kappa_s$, using Lemma~\ref{fact:2.2} we see that
\begin{align*}
\phi_k(Q_s y+b)
& = 
\phi_k(Q_s y)
+ O \big( 2^{-3k/4} \eta_{\le2\delta k}(2^{-k}\circ y) \big),  \\
S_k(\xi+a/q)
& =
G(a/q)J_k(\xi) + O(2^{-k/2}), \\
\ex\big((Q_s y+b)\cdot\xi\big)
& =
\ex\big((Q_s y)\cdot\xi\big) + O( 2^{-3k/4} ),
\end{align*}
uniformly in $k\ge \kappa_s$, $s \ge 0$, $a/q\in R_s^\Gamma$, $|2^k \circ \xi| \lesssim 2^{\delta k}$ and $|b| \lesssim Q_s$. These identities lead easily to \eqref{eq:20.1}.
Therefore for $x \in \ZZ^\Gamma$ and $a \in \ZZ^\Gamma_{Q_{s}}$, we have  
\begin{align*}
f\ast K_{k,s}(Q_s x+a)
&=\sum_{y\in \ZZ^\Gamma}\sum_{b\in \ZZ^\Gamma_{Q_{s}}}K_{k,s}\big(Q_s(x-y)+a-b\big) f(Q_s y+b)
\\
&=
\sum_{y\in \ZZ^\Gamma}\sum_{b\in \ZZ^\Gamma_{Q_{s}}}W_{k,Q_s}\big(Q_s(x-y)\big)V_s(a-b) f(Q_s y+b)
\\
&\quad+
\sum_{y\in \ZZ^\Gamma}\sum_{b\in \ZZ^\Gamma_{Q_{s}}}E_{k,s}(x-y, a-b) f(Q_s y+b).
\end{align*}
Letting 
$$
F_s(a, Q_s y) 
:= 
\sum_{b\in \ZZ^\Gamma_{Q_{s}}} V_s(a-b) f(Q_s y+b)=g_y\ast_{\ZZ^\Gamma_{Q_{s}}}V_s(a),
$$
where $g_y(b):=f(Q_s y+b)$ for $b \in \ZZ^\Gamma_{Q_{s}}$ and then $g_y$ is extended $Q_{s}$ periodically to $\ZZ^\Gamma$. We can write
$$
\sum_{y\in \ZZ^\Gamma}\sum_{b\in \ZZ^\Gamma_{Q_{s}}}W_{k,Q_s}\big(Q_s(x-y)\big)V_s(a-b) f(Q_s y+b)=F_s(a, \cdot)\ast_{Q_s \ZZ^\Gamma} W_{k, Q_s}(Q_s x),
$$
and consequently
\begin{align} \label{id:606}
	\begin{split}
\sup_{k\ge \kappa_s} 
|f\ast K_{k,s}(Q_{s} x + a)|
& \le 
\sup\limits_{\substack{2^{\delta k}\ge Q_s \\ k\ge D}}
|F_s(a, \cdot)\ast_{Q_s \ZZ^\Gamma} W_{k, Q_s}(Q_s x)| \\
& \quad +
\sum_{k\ge \kappa_s} \sum_{y\in \ZZ^\Gamma}\sum_{b\in \ZZ^\Gamma_{Q_{s}} }
|E_{k,s}(x-y, a-b)| |f(Q_s y+b)|.
\end{split}
\end{align}
Notice that \eqref{id:605} implies $\sup_{|b| \lesssim Q_{s}} \| E_{k,s}(\cdot, b) \|_{\ell^{1} (\ZZ^\Gamma)} \lesssim 2^{-k/4}$ and using this we see that the $\ell^{2} (\ZZ^\Gamma)$ norm of the the second term in \eqref{id:606} is controlled by
 \begin{align*}
 \sum_{k\ge \kappa_s} Q_{s}^{|\Gamma|}  2^{-k/4} \|f\|_{\ell^2(\ZZ^\Gamma)}
 \lesssim
  2^{-\kappa_s/8} \|f\|_{\ell^2(\ZZ^\Gamma)}
  \le
 2^{-s/8} \|f\|_{\ell^2(\ZZ^\Gamma)}.
 \end{align*}
On the other hand, the $\ell^{2} (\ZZ^\Gamma)$ norm of the first term in \eqref{id:606} is bounded by
\begin{align*}
&	\Big\|
	\big\|
	\sup\limits_{\substack{2^{\delta k}\ge Q_s \\ k\ge D}}
	|F_s(a, \cdot)\ast_{Q_s \ZZ^\Gamma} W_{k, Q_s}(Q_s x)|
	\big\|_{\ell^2(x)}
	\Big\|_{\ell^2(a)} \\
& \qquad \le
\big\| \mathcal W_{*,Q_s} \big\|_{\ell^2(Q_s\ZZ^\Gamma)\rightarrow \ell^2(Q_s\ZZ^\Gamma)}
\Big\|
\big\|
F_s(a, Q_{s} x)
\big\|_{\ell^2(x)}
\Big\|_{\ell^2(a)}.
\end{align*}
Further, notice that 
\begin{align*}
	\Big\|
	\big\|
	F_s(a, Q_{s} x)
	\big\|_{\ell^2(x)}
	\Big\|_{\ell^2(a)}
&	=
	\Big\|
	\big\|
	g_x\ast_{\ZZ^\Gamma_{Q_{s}}} V_s(a)
	\big\|_{\ell^2(a)}
	\Big\|_{\ell^2(x)} \\
&	\le
	\|\mathcal V_s\|_{\ell^2(\ZZ^\Gamma_{Q_{s}})\rightarrow \ell^2( \ZZ^\Gamma_{Q_{s}} )}
\big\| \| g_x \|_{\ell^2(a)} \big\|_{\ell^2(x)}.
\end{align*}
Since we have $	\big\| \| g_x \|_{\ell^2(a)} \big\|_{\ell^2(x)} = \|f\|_{\ell^2(\ZZ^\Gamma)}$, 
Lemma~\ref{lem:3.1} follows.
\end{proof}

We pass to the proof of Lemma \ref{lem:3.2}.

\begin{proof}[Proof of Lemma \ref{lem:3.2}]
Here we will use a high order $TT^*$ argument. 
Since 
$\mathcal{V}_s(f) =f\ast_{\ZZ^\Gamma_{Q_{s}} }V_s$, we have
$$
\big(\mathcal{V}_s^* \mathcal{V}_s\big)^r f =: f\ast_{ \ZZ^\Gamma_{Q_{s}} }V_s^r,
$$
where as in Section \ref{subsec:TT*} we have
\begin{equation*}
V_s^r(y)
=
\sum_{h_1, g_1, \dots, h_r, g_r\in \ZZ^\Gamma_{Q_{s}} }\Big(\prod_{j=1}^r \overline{V_s(h_j)} V_s(g_j)\Big)
\ind{\{0\}}\Big(y-\sum_{j=1}^r (g_j - h_j)\Big), \qquad y \in \ZZ^\Gamma_{Q_{s}}.
\end{equation*}
To prove Lemma \ref{lem:3.2} it suffices to show that
\begin{equation}\label{eq:22.1}
\|V_s^r\|_{\ell^1(\ZZ^\Gamma_{Q_{s}})}\lesssim 2^{-s}, \qquad s\ge 0.
\end{equation}
Since 
$$
\ind{\{0\}}(z)
=
Q_s^{-|\Gamma|}\sum_{b \in \ZZ^\Gamma_{Q_{s}}} \ex \big(z\cdot b/Q_s\big),
\qquad z \in \ZZ^\Gamma_{Q_{s}},
$$
we see that
\begin{align*}
V_s^r(y)
&=
Q_s^{-|\Gamma|}\sum_{b\in \ZZ^\Gamma_{Q_{s}}} \ex(y\cdot b/Q_s)\sum_{h_1, g_1, \dots, h_r, g_r\in \ZZ^\Gamma_{Q_{s}}}
\Big(\prod_{j=1}^r \overline{V_s(h_j)} \ex(h_j\cdot b/Q_s) V_s(g_j)
\ex(-g_j\cdot b/Q_s)\Big)
\\
&=
Q_s^{-|\Gamma|}\sum_{b\in \ZZ^\Gamma_{Q_{s}}} \ex(y\cdot b/Q_s)
\Big|\sum_{g\in \ZZ^\Gamma_{Q_{s}} } V_s(g) \ex(-g\cdot b/Q_s)\Big|^{2r}.
\end{align*}
To prove \eqref{eq:22.1} it suffices to show that
\begin{equation}\label{eq:23.1}
\sum_{b \in \ZZ^\Gamma_{Q_{s}}}
\Big|\sum_{g\in \ZZ^\Gamma_{Q_{s}} } V_s(g) \ex(-g\cdot b/Q_s)\Big|^{2r}
\lesssim 2^{-s},\qquad s\ge 0.
\end{equation}
This, however, follows easily since we have 
\begin{equation*}
\sum_{g\in \ZZ^\Gamma_{Q_{s}} } V_s(g) \ex(-g\cdot b/Q_s)
=
\ind{R_s^\Gamma \cap [0,1)^\Gamma }(b/Q_s)G(b/Q_s)
\end{equation*}
and consequently the left-hand side of \eqref{eq:23.1} is controlled by
$$
\sum_{a/q\in R_s^\Gamma \cap [0,1)^\Gamma} |G(a/q)|^{2r}\lesssim 2^{-s},
$$
by our choice of $r$, see \ref{item2}. This proves Lemma \ref{lem:3.2} as desired.
\end{proof}

It remains to prove Lemma \ref{lem:3.3}.

\begin{proof}[Proof of Lemma \ref{lem:3.3}]
For $0\le \omega\le k$, we first define
\begin{align*}
 W_{k, \omega, Q}(x):=Q^{|\Gamma|}\phi_k(x)\int_{\RR^\Gamma} \eta_{\le\delta\omega}(2^k\circ \xi) \ex(x\cdot\xi)J_k(\xi)\,d\xi,\qquad x\in Q\ZZ^\Gamma,
\end{align*}
and the continuous counterpart
\begin{align} \label{id:611}
\widetilde{W}_{k, \omega}(x)
:=
\phi_k(x)\int_{\RR^\Gamma} \eta_{\le\delta\omega}(2^k\circ \xi) \ex(x\cdot\xi)J_k(\xi)\,d\xi,\qquad x\in \RR^\Gamma.
\end{align}
By \eqref{id:607}, note that 
$W_{k, Q}=W_{k,k, Q}$,
and thus we have
\begin{equation*}
 W_{k, Q}=W_{k,0, Q}+\sum_{\omega=0}^{k-1}\big(W_{k, \omega+1, Q}-W_{k, \omega, Q}\big).
\end{equation*}
Therefore it suffices to prove that 
\begin{equation}\label{eq:(*)}
\big\|\sup\limits_{\substack{k\ge D\\2^{\delta k}\ge Q}}
|f\ast_{Q\ZZ^\Gamma} W_{k, 0, Q}| \big\|_{\ell^2(Q\ZZ^\Gamma)}
\lesssim 
\|f\|_{\ell^2(Q\ZZ^\Gamma)}, \qquad Q \ge 1,
\end{equation}
and 
\begin{equation}\label{eq:(**)}
\Big\|\sup\limits_{\substack{k\ge D\\2^{\delta k}\ge Q}} \Big|f\ast_{Q\ZZ^\Gamma} \sum_{\omega=0}^{k-1}\big(W_{k, \omega+1, Q}-W_{k, \omega, Q}\big) \Big| \Big\|_{\ell^2(Q\ZZ^\Gamma)}
\lesssim 
\|f\|_{\ell^2(Q\ZZ^\Gamma)}, \qquad Q \ge 1.
\end{equation}
Dominating inequality \eqref{eq:(**)} by a square function and by
using Khintchine's inequality, inequality \eqref{eq:(**)} will follow from the
following estimate
\begin{equation}\label{eq:(***)}
\Big\|\sum\limits_{\substack{k\ge D\\2^{\delta k}\ge Q\\k>\omega}}\chi_k f\ast_{Q\ZZ^\Gamma} \big(W_{k, \omega+1, Q}-W_{k, \omega, Q}\big) \Big\|_{\ell^2(Q\ZZ^\Gamma)}
\lesssim 
2^{-\omega/D}\|f\|_{\ell^2(Q\ZZ^\Gamma)},
\end{equation}
uniformly in $\omega\in\NN$, $Q\ge 1$ and $|\chi_k|\le 1$ for all
$k\in\ZZ_+$.

We begin with reducing the proof of \eqref{eq:(*)} and \eqref{eq:(***)} to showing analogous results in the continuous setup. 
Let us define 
\begin{align*}
\Upsilon_\omega(x) 
& :=
\eta_{\le \delta (\omega+1)}(x)-\eta_{\le \delta \omega}(x), 
\qquad x \in \RR^{\Gamma}, \quad \omega \in \NN, \\
S_{k, \omega}(x)
& :=
\Big( \prod_{\gamma\in\Gamma}2^{-k|\gamma|} \Big)
\frac{1}{|\Omega|}\int_{\Omega}\widehat{\Upsilon_\omega}\big((u)^\Gamma-2^{-k}\circ x \big)\,du, 
\qquad x \in \RR^{\Gamma}, \quad 0 \le \omega \le k, \\
\rho_{k, \omega}(x)
& :=
\Big( \prod_{\gamma\in\Gamma} 2^{-k|\gamma|+\delta\omega} \Big) 
\int_\Omega \big(1+2^{\delta \omega} |(u)^\Gamma-2^{-k}\circ x| \big)^{-D}\,du, 
\qquad x \in \RR^{\Gamma}, \quad 0 \le \omega \le k,
\end{align*}
where $\widehat{\Upsilon_\omega}:=\mathcal F_{\RR^{\Gamma}}\Upsilon_\omega$.
Observe that we have 
\begin{align} \label{eq:bb}
	\begin{split}
	\widetilde{W}_{k, \omega+1}(x)-\widetilde{W}_{k, \omega}(x)
	& =
	\phi_k(x)S_{k, \omega}(x), \qquad x \in \RR^{\Gamma}, \\ 
		\widehat{S_{k,\omega}}(\theta):=\mathcal F_{\RR^{\Gamma}}S_{k,\omega}(\theta)
	& =
	\Upsilon_\omega(2^k\circ \theta)J_k(\theta), \qquad \theta \in \RR^{\Gamma},
\end{split}
\end{align}
where $J_k$ defined in \eqref{eq:9.1}. Further, we have the estimate
\begin{align}\label{eq:Sdec}
|S_{k, \omega}(x)|+\sum_{\gamma\in\Gamma}2^{k|\gamma|-\delta\omega} |\partial_{x_\gamma} S_{k, \omega}(x)|
\lesssim 
\rho_{k, \omega}(x), \qquad  x \in \RR^{\Gamma}, \quad 0 \le \omega \le k,
\end{align}
and consequently
\begin{align} \label{id:608}
\|\rho_{k, \omega}\|_{L^1(\R^\Gamma)}
\simeq 
\| \rho_{k, \omega} \|_{\ell^1(\Z^\Gamma)}\simeq 1, 
\qquad
\| (1 + |2^{-k} \circ (\cdot)|)\rho_{k, \omega}\|_{L^1(\R^\Gamma)}
\lesssim
1.
\end{align}

We begin with showing that \eqref{eq:(*)} and \eqref{eq:(***)} follow from
\begin{align}
	\label{eq:delta0}
	\big\| \sup_{k\ge 0} |F\ast_{\R^\Gamma} \widetilde{W}_{k, 0}| \big\|_{L^2(\R^\Gamma)}
	&	\lesssim 
	\|F\|_{L^2(\R^\Gamma)}, \qquad F \in L^2(\RR^\Gamma), \\
	\label{eq:delta}
\Big\|\sum\limits_{k>\omega}\chi_k F\ast_{\RR^\Gamma} S_{k, \omega} \Big\|_{L^2(\RR^\Gamma)}
&	\lesssim 
	2^{-\omega/D}\|F\|_{L^2(\RR^\Gamma)}, \qquad F \in L^2(\RR^\Gamma),
	\end{align}
uniformly in $\omega\in\NN$ and $|\chi_k|\le 1$ for all $k\in\NN$.

For a given $f \in \ell^2(Q\ZZ^\Gamma)$ we 
consider
 $F\colon \RR^\Gamma\rightarrow \CC$ given by 
 \begin{equation*}
 	F(x) = F_{f}(x)
 	:=
 	\sum_{n\in Q\ZZ^\Gamma}\mathbbm{1}_{n+[0,Q)^\Gamma}(x) f(n).
 \end{equation*}
 Notice that $\|F\|_{L^2(\R^\Gamma)} = Q^{|\Gamma|/2} \|f\|_{\ell^2(Q\ZZ^\Gamma)}$.
Further, taking into account \eqref{id:611}, we obtain
\begin{align*}
\big| \widetilde{W}_{k, \omega}(x + z) - \widetilde{W}_{k, \omega}(x) \big|
\lesssim
E_{k} (x),
\end{align*}
uniformly in $x \in \RR^{\Gamma}$, $|z| \lesssim Q \le 2^{\delta k}$ and 
$0 \le \omega \le k$, where
\begin{align*}
	E_{k} (x)
	: =
	2^{-k/2} \Big( \prod_{\gamma\in\Gamma} 2^{-k|\gamma|} \Big)
	\eta_{\le 2\delta k}(2^{-k} \circ x),
	\qquad x \in \RR^{\Gamma}, \quad k \ge 0. 
\end{align*}
This shows that
\begin{align*}
F\ast_{\R^\Gamma}  \widetilde{W}_{k, \omega}(x + z) 
& =
\sum_{y \in Q\ZZ^{\Gamma} } \int_{[0,Q)^{\Gamma}} F(y + u) 
\widetilde{W}_{k, \omega}(x - y + z - u) \, du \\
& =
f \ast_{Q\ZZ^\Gamma}  W_{k, \omega,Q} (x)
+ O(Q^{|\Gamma|} |f| \ast_{Q\ZZ^\Gamma} E_{k} (x) ),
\end{align*}
uniformly in $x \in Q\ZZ^{\Gamma}$, $|z| \lesssim Q \le 2^{\delta k}$ and 
$0 \le \omega \le k$.
Squaring this identity and averaging in $z \in [0,Q)^{\Gamma}$ and then using the fact that
$\sum_{k \ge \omega} 2^{k/4} \|E_{k} \|_{\ell^1(Q\ZZ^\Gamma)} \lesssim 2^{-\omega/8}$ 
and \eqref{eq:bb},
we see that
\begin{align*}
\big\|\sup\limits_{\substack{k\ge D\\2^{\delta k}\ge Q}}
|f\ast_{Q\ZZ^\Gamma} W_{k, 0, Q}| \big\|_{\ell^2(Q\ZZ^\Gamma)}^2
& \lesssim 
Q^{-|\Gamma|} \big\| \sup_{k\ge 0} |F\ast_{\R^\Gamma} \widetilde{W}_{k, 0}| \big\|_{L^2(\R^\Gamma)}^2
+
\|f\|_{\ell^2(Q\ZZ^\Gamma)}^2, \\
\Big\|\sum\limits_{\substack{k\ge D\\2^{\delta k}\ge Q\\k>\omega}}
\chi_k f\ast_{Q\ZZ^\Gamma} \big(W_{k, \omega+1, Q}-W_{k, \omega, Q}\big) \Big\|_{\ell^2(Q\ZZ^\Gamma)}^2
& \lesssim 
Q^{-|\Gamma|}
\Big\| \sum\limits_{\substack{k\ge D\\2^{\delta k}\ge Q\\k>\omega}} 
\chi_k F\ast_{\RR^\Gamma} (\phi_k S_{k, \omega}) \Big\|_{L^2(\RR^\Gamma)}^2 \\
& \quad 
+
2^{-\omega/4}\|f\|_{\ell^2(Q\ZZ^\Gamma)}^2.
\end{align*}
Since $\| (1 - \phi_k) S_{k, \omega} \|_{L^1(\RR^\Gamma)} \lesssim 2^{-k}$, we see that \eqref{eq:(*)} and \eqref{eq:(***)} follow from \eqref{eq:delta0} and \eqref{eq:delta} respectively.

We begin with treating \eqref{eq:delta0}. It is straightforward to notice that 
\begin{equation*}
|\widetilde{W}_{k, 0}(x)|
\lesssim 
\Big( \prod_{\gamma\in\Gamma} 2^{-k|\gamma|} \Big) 
\big(1+ |2^{-k}\circ x| \big)^{-D},
\qquad x \in \RR^{\Gamma}, \quad k \ge 0.
\end{equation*}
Let $\rho_{\Gamma}$ be a quasi-metric on $\RR^{\Gamma}$ given by $\rho_{\Gamma} (x,y) := \sup_{\gamma \in \Gamma} |x_{\gamma} - y_{\gamma}|^{1/|\gamma|}$. Then
$(\RR^{\Gamma}, \rho_{\Gamma}, |\cdot|)$, where $|\cdot|$ is the Lebesgue measure on $\RR^{\Gamma}$, is a space of homogeneous type. Let us denote the corresponding 
Hardy--Littlewood maximal function by $M_{\rho_{\Gamma}}$. Then using the fact that 
\begin{align*}
	1 + |x-y| 
	\lesssim
	(1 + \rho_{\Gamma} (x,y))^{d_{\Gamma}}, 
	\qquad
	1 + \rho_{\Gamma} (x,y)
	\lesssim
	1 + |x-y|, 
	\qquad x,y \in \RR^{\Gamma},
\end{align*}
we can easily deduce that 
$|F\ast_{\R^\Gamma} \widetilde{W}_{k, 0}| \lesssim M_{\rho_{\Gamma}} F$.
Thus \eqref{eq:delta0} follows from boundedness of $M_{\rho_{\Gamma}}$ on $L^2(\RR^\Gamma)$.

We turn to proving \eqref{eq:delta}.
The Cotlar--Knapp--Stein lemma reduces \eqref{eq:delta} to proving
\begin{equation}\label{eq:2delta}
\|\mathcal{S}_{j, \omega}\mathcal{S}_{k, \omega}^*\|_{L^2(\RR^\Gamma)\rightarrow L^2(\RR^\Gamma)}
\lesssim 
2^{-2\omega/D}2^{-|k-j|/D},
\end{equation}
uniformly in $0\le \omega<j\le k$, where $\mathcal{S}_{m, \omega}f:= f \ast S_{m, \omega}$. We now distinguish two cases.

\paragraph{\textbf{Case 1}} Assume first that $\omega\le (k-j)/(2\delta)$. Since the term $k-j$ dominates $\omega$, it suffices to show that
\begin{equation*}
\|S_{k, \omega}^* \ast S_{j, \omega}\|_{L^1(\R^\Gamma)}\lesssim 2^{-(k-j)/2}.
\end{equation*}

Noting that  $\int_{\RR^\Gamma} S_{k, \omega}(x)\, dx=\widehat{S_{k, \omega}}(0)=0$,  and using \eqref{eq:Sdec}, we obtain
\begin{align*}
	|S_{k, \omega}^* \ast S_{j, \omega} (x)|
	& \le 
	\int_{\RR^\Gamma} |S_{j, \omega}(y)| 
	\big|S_{k, \omega}(y-x)-S_{k, \omega}(-x)\big|\,dy
	\\
	&\lesssim
	\int_{\RR^\Gamma} \rho_{j, \omega}(y) \int_{0}^1 
	\big| \partial_{t} \big( S_{k, \omega}(ty-x) \big) \big| \,dt \,dy
	\\
	&\lesssim
	2^{-(k-j) + \delta \omega}\int_{\RR^\Gamma}\rho_{j, \omega}(y) |2^{-j}\circ y|
	\int_{0}^1 \rho_{k, \omega}(ty-x)\,dt\,dy.
\end{align*}
Now using Fubini--Tonelli's theorem and \eqref{id:608} twice (first to the $x$ variable and then to the $y$ variable) we obtain 
$\|S_{k, \omega}^* \ast S_{j, \omega}\|_{L^1(\R^\Gamma)} \lesssim 2^{-(k-j) + \delta \omega} \le 2^{-(k-j)/2}$ and Case 1 follows.

\paragraph{\textbf{Case 2}} It remains to consider $\omega> (k-j)/(2\delta)$. In this case, we are seeking a  decay in $\omega$. The application of a high order $TT^*$ argument reduces the problem to showing 
\begin{equation} \label{id:610}
\big\| \big(\mathcal{S}_{k, \omega}^*\mathcal{S}_{k, \omega}\big)^r \big\|_{L^2(\RR^\Gamma)\rightarrow L^2(\RR^\Gamma)}
\lesssim 
2^{-\omega/(2\delta)}. 
\end{equation}
Using the inverse Fourier transform theorem and \eqref{eq:bb} we see that the kernel of the above operator is 
\begin{align*} 
K_{k, \omega}^{r}(x)
=
\int_{\RR^\Gamma} \ex(\theta\cdot x) 
\big| \widehat{S_{k, \omega}}(\theta)\big|^{2r}\,d\theta
=
\int_{\RR^\Gamma} \ex(\theta\cdot x) 
\big| \Upsilon_\omega(2^k\circ \theta)J_k(\theta) \big|^{2r}\,d\theta.
\end{align*}
Changing the variable $2^k \circ \theta \mapsto \theta$, we may further write
\begin{align*} 
	K_{k, \omega}^{r}(x)
=
\Big(\prod_{\gamma\in\Gamma}2^{-k|\gamma|} \Big)
	\int_{\RR^\Gamma} \ex\big( \theta\cdot (2^{-k} \circ x) \big) 
	\big| \Upsilon_\omega(\theta)J(\theta) \big|^{2r}\,d\theta.
\end{align*}
Using  Lemma~\ref{lem:A2}, we see that for any fixed $\alpha \in \NN^{\Gamma}$, we have 
\begin{align*} 
	\big|D^\alpha \big( \Upsilon_\omega(\theta)J(\theta) \big)\big|
	\lesssim_\alpha 
	(1+|\theta|)^{-1/d_{\Gamma}} \ind{|\theta| \ge 2^{\delta \omega}}, 
	\qquad \theta \in \RR^{\Gamma}, \quad \omega \in \NN.
\end{align*}
Consequently, integrating by parts, we obtain $\| K_{k, \omega} \|_{L^1(\RR^\Gamma)} \lesssim 2^{-\omega/(2\delta)}$, since
\begin{align*} 
|K_{k, \omega}^{r}(x)|
&	\lesssim
	\Big(\prod_{\gamma\in\Gamma}2^{-k|\gamma|} \Big)
	\big(1+|2^{-k}\circ x| \big)^{-D}
	\int_{\RR^\Gamma} (1+|\theta|)^{-1/\delta^{2}}  
	\ind{|\theta| \ge 2^{\delta \omega}} \,d\theta \\
&	\lesssim
	2^{-\omega/(2\delta)}
		\Big(\prod_{\gamma\in\Gamma}2^{-k|\gamma|} \Big)
	\big(1+|2^{-k}\circ x| \big)^{-D},
\end{align*}
uniformly in $x\in\R^\Gamma$. Hence \eqref{id:610} follows completing the proof of Lemma~\ref{lem:3.3}.
\end{proof}

\end{document}